\documentclass{article}

\usepackage[english]{babel}

\usepackage[letterpaper,top=3cm,bottom=3cm,left=3cm,right=3cm,marginparwidth=1.75cm]{geometry}
 \usepackage{enumitem}
\usepackage{amsmath}
\usepackage{amssymb}
\usepackage{placeins}
\usepackage{graphicx}
\usepackage{cite}
\usepackage[colorlinks=true, allcolors=blue]{hyperref}
\usepackage[noblocks]{authblk}
\usepackage{algorithm}

\usepackage{algpseudocode}   
\usepackage{float}          
\usepackage{algpseudocode}
\usepackage{amsthm}
\usepackage{arydshln}
\newlength\figureheight
\newlength\figurewidth
\usepackage{amssymb}
\usepackage{mathdots}
\usepackage{multirow}
\usepackage{float}
 \usepackage{bm}
\usepackage{xcolor}
\usepackage{algpseudocode}
\usepackage{mathtools}
\usepackage{tikz}
\usepackage{tikzscale}
\usetikzlibrary{matrix}
\usepackage{pgfplots}
\usepackage{pdfpages}
\usepackage{hyperref}
\usepackage{mathtools}
\usepackage[noblocks]{authblk}
\usepackage{verbatim}
\usepackage{tikz}
\usetikzlibrary{plotmarks}
\usepackage{subcaption}

\usepackage{tablefootnote}

\usepackage{tikz}
\usetikzlibrary{matrix,calc}
\usepackage{sidecap}

\hypersetup{
	colorlinks=true,
	linkcolor=blue,      
}
\usetikzlibrary{arrows.meta,calc,
	tikzmark,pgfplots.groupplots,external}
\pgfplotsset{compat=1.18}

\tikzexternaldisable

\usepackage{url}
\usepackage{amsthm}
\newtheorem{theorem}{Theorem}[section]
\newtheorem{lemma}[theorem]{Lemma}

\newtheorem{problem}{Problem}[section]
\newtheorem{proposition}{Proposition}[section]

\newtheorem{example}{Example}[section]

\pgfplotsset{compat=1.18}
\begin{document}
	\title{Updating and Downdating the Recurrences of Discrete Multiple Orthogonal Polynomials on the Real Line}
	\providecommand{\keywords}[1]{\textit{Keywords: } #1}
	\date{\today}

	\author[1a*]{Amin Faghih}
	\author[1b]{Marc Van Barel}
	\author[1c]{Raf Vandebril}
	\affil[1]{Department of Computer Science, KU Leuven, Leuven, Belgium}

\affil[a*]{Corresponding author: Amin Faghih \texttt{amin.faghih@kuleuven.be}}
\affil[b]{ \texttt{marc.vanbarel@kuleuven.be}}
\affil[c]{ \texttt{raf.vandebril@kuleuven.be}}

	\renewcommand*{\Affilfont}{\small\it} 

	\maketitle	
\begin{abstract}
We study multiple orthogonal polynomials with orthogonality defined by multiple positive discrete measures on the real line. Focusing on type~I and type~II multiple orthogonal polynomials and their step-line recurrence relations, we consider the reconstruction of the associated banded upper Hessenberg recurrence matrix from given nodes and weights of the measures. We propose efficient updating and downdating procedures that modify an existing recurrence matrix when nodes are added to or removed from the discrete measures. The updating strategy is based on solving an inverse eigenvalue problem, while the downdating procedure relies on an eigenvalue deflation technique inspired by a QR-type of algorithm. Numerical experiments confirm the stability and efficiency of the proposed approaches.
\end{abstract}
\\
\begin{small} {\textbf{AMS subject classification: }} 33C45, 15A63, 65F18, 15B99.
	\end{small}
	\renewcommand{\thefootnote}{\fnsymbol{footnote}}

	\section{Introduction}
Orthogonal polynomials play a fundamental role in the design and analysis of numerical methods \cite{Gautschi04,Liesenn,OlSlTo20}. Typically, for classical families, such as Jacobi polynomials, explicit recursion formulas are available and allow their direct construction \cite{MR2542683}. When the orthogonality is defined through a more general inner product such as 
\[
\langle f, g \rangle_N = \sum_{i=1}^{N} |\alpha_i|^2\,  f(z_i)\,\overline{g(z_i)},
\]
with nodes $z_i\in \mathbb{C}$ and weights $\alpha_i \in \mathbb{C}$, which does not correspond to a classical family, one typically relies on numerical linear algebra–based approaches to compute the associated orthogonal polynomials in a stable and efficient manner. These orthogonal polynomials can be described through a recurrence relation whose coefficients can be grouped together in a Hessenberg matrix; for more information, we refer to the book of Bultheel and Van Barel \cite{MarcBook}. One approach for generating such polynomials with respect to a positive semi-definite inner product relies on manipulating this recurrence matrix directly. The solution of this problem is often referred to as solving an inverse eigenvalue problem (IEP) \cite{MR928047,Chu98}. In particular, a method for computing this Hessenberg recurrence matrix is based on repeatedly adding nodes until all nodes defining the inner product have been incorporated, that is, an updating procedure.

In the special case of orthonormal polynomials with real nodes $z_i$, Gragg and Harrod \cite{GraggH} proposed an algorithm for constructing the associated Jacobi matrix, which is Hermitian and tridiagonal. This procedure originates from earlier work by Rutishauser \cite{Ru} and offers a significant computational advantage compared with methods designed for general nodes $z_{i}$. When the nodes lie on the unit circle, $z_i \in \mathbb{C}$ with $|z_{i}|=1$, a related strategy yields a unitary Hessenberg matrix rather than a Jacobi matrix \cite{MR1102398}. The core idea of this approach has been used to develop updating procedures for several IEPs \cite{VanbVanbVand22,Vanb23,Faghih2025Sobolev}.

Instead of adding a node to the inner product, one may also consider the removal of a node, leading to what is commonly referred to as a downdating problem. Such downdating procedures are inherently more challenging to carry out in a numerically stable manner \cite{MR1089163}. Results concerning the downdating of orthogonal polynomials in the case where the nodes $z_i$ lie on the unit circle can be found in Ammar, Gragg, and Reichel \cite{MR1168509}. The eigenvector deflation technique of Mastronardi and Van Dooren \cite{MR3867618} for removing a prescribed eigenvalue from a Hessenberg matrix provides insight into the design of stable downdating strategies for orthogonal polynomial recurrences. An alternative approach is based on a QR-type technique, originally developed for unitary Hessenberg matrices, which isolates a known eigenvalue and was introduced by Ammar et al. \cite{MR1150072}. Van Barel, Van Buggenhout, and Vandebril \cite{MR4731119} employed both eigenvector and QR-type algorithms to downdate the recurrences of orthogonal polynomials and rational functions. In practice, updating and downdating are used, for example, in recursive least squares and sliding-window least squares, where data points are incrementally added and discarded, allowing matrix factorizations to be modified efficiently instead of recomputed \cite{GolubVanLoan2013,Bjorck1996,FaVBVBVa24}.

In this paper, we generalize these updating and downdating procedures to the case of multiple orthogonal polynomials (MOPs). Multiple orthogonal polynomials generalize standard orthogonal polynomials by requiring orthogonality with respect to $r$ measures $\omega_{1},\omega_{2},\ldots,\omega_{r}$ \cite{MR1662713,MR1458825,MR2542683,MR2358391,MR1808581}. In this text, we deal with positive discrete measures on $\mathbb{R}$ \cite{MR1985676}
 \begin{equation}\label{MEASURE}
    \omega_j = \sum_{i=1}^{N} \alpha_{j,i} \delta_{z_{i}}, \; \; \alpha_{j,i} > 0, \; z_{i} \in \mathbb{R}, \; N \in \mathbb{N}, \; 1 \le j \le r,
\end{equation}
where the nodes $z_i$ are distinct and $\delta_{z_i}$ denotes the Dirac measure at $z_i$. The support of each discrete measure $\omega_j$ is the closure of $\{z_i\}_{i=1}^N$, and we denote by $\Delta$ the smallest closed interval in $\mathbb{R}$ containing $\{z_i\}_{i=1}^N$. Let $\bm{n}=(n_1,\ldots,n_r)\in\mathbb{N}^r$ be a multi-index, and define its total order by $|\bm{n}|=\sum_{j=1}^r n_j$. Multiple orthogonal polynomials are commonly divided into two classes. Type~I MOPs are given by a vector of polynomials
$(A_{\bm{n},1},\ldots,A_{\bm{n},r})$, where each component satisfies
$\deg A_{\bm{n},j}= n_j-1$. These polynomials are characterized by the orthogonality conditions
\begin{equation}\label{type1}
\sum_{i=1}^N \Big( \sum_{j=1}^r \alpha_{j,i}\, A_{\bm{n},j}(z_{i})\Big) z_{i}^k  =0,
\qquad k=0,1,\ldots,|\bm{n}|-2.
\end{equation}
A type~II multiple orthogonal polynomial associated with $\bm{n}$ is a single polynomial
$P_{\bm{n}}$ of degree $|\bm{n}|$ satisfying
\begin{equation*}
\sum_{i=1}^N \alpha_{j,i}\, P_{\bm{n}}(z_{i}) z_{i}^k =0,
\qquad k=0,1,\ldots,n_j-1,\quad j=1,\ldots,r.
\end{equation*}
If we have an AT-system\footnote{An AT-system is a particular set of measures satisfying the algebraic Chebyshev property; see, e.g., \cite{MR1985676,CouVanass05,MR2542683}.} with $r$ positive discrete measures $\omega_{j}$ supported on $\Delta$, then for $|\bm{n}| \leq N$, the associated type~I and type~II MOPs are uniquely determined up to a scalar multiplicative constant. Moreover, when $|\bm n|=N$, the zeros of the type~II polynomial coincide with the prescribed nodes $\{z_i\}_{i=1}^N$. Detailed assumptions and proofs can be found in \cite[Section 2]{Faghih2025KrylovCore} and the references therein.

Multiple orthogonal polynomials admit several recurrence relations, depending on the chosen path of multi-indices in $\mathbb{N}^r$ \cite{CouVanass05,Van-nearestneighbor,FilHanVanass15,Vana24,MR2542683}. In this work, we concentrate on the \emph{step-line} recurrence (see Section \ref{Sec2}), an $(r+2)$-term relation that gives rise to an $(r+2)$-banded upper Hessenberg matrix including the recurrence coefficients \cite{MR1985676,Faghih2025KrylovCore}. In this work, given the nodes and weights of
some discrete measures, we aim to reconstruct the corresponding recurrence matrix. These nodes and weights are linked to the eigenvalues and eigenvectors of the recurrence matrix, making this an instance of an IEP \cite{MR928047,Chu98}. 

Recently, Faghih, Rinelli, Van Barel, Vandebril, and Vermeiren~\cite{Faghih2025KrylovCore} established the link between type II multiple orthogonal polynomials and standard Krylov subspaces, as well as between type I MOPs and block Krylov subspaces. Building on this connection, they developed a Lanczos-type algorithm with multiple starting vectors to generate biorthogonal pairs of type I and type II discrete MOPs, addressing the corresponding inverse eigenvalue problem. This specialized variant of the Lanczos approach was introduced by Aliaga, Boley, Freund, and Hern\'andez~\cite{MR1665942}. They also implemented an efficient alternative approach based on core transformations. 

In this paper, we propose updating and downdating procedures which are more memory efficient than the Lanczos and core transformation methods, since they do not require storing the entire basis. Furthermore, the updating and downdating procedures are more flexible, as they allow the reuse of existing recurrence relations to construct new recurrences under modest modifications of the inner product, such as adding or deleting nodes. This flexibility is not possible with other approaches, which would require a costly restart from scratch. Updating procedures start from an available solution of an IEP and efficiently compute the solution to an IEP where the discrete inner product underlying the biorthogonal type I and type II MOPs is enlarged by one node. Conversely, downdating removes a node and its weights, producing a recurrence matrix of one dimension lower. 

The updating procedure in this research is primarily based on similarity transformations performed on the Hessenberg recurrence matrix \cite{MR3227393}. The downdating strategy is formulated using an eigenvalue deflation technique \cite{MR3867618}.
Numerically, the computation of MOP bases via the bi-orthogonal Lanczos and core transformation algorithms \cite{Faghih2025KrylovCore} are typically so ill-conditioned that the computation of MOPs of large degree becomes infeasible in double-precision arithmetic; see \cite[Section~6]{Faghih2025KrylovCore}. The proposed scaling strategy, together with the nature of the updating procedure, substantially improves the conditioning and enables the computation of larger sequences of MOPs than those reported in \cite{Faghih2025KrylovCore}.

The remainder of this paper is organized as follows. In Section~\ref{SSec2}, we review fundamental properties of MOPs, including their recurrence relations and biorthogonality structure. We then reformulate the construction of biorthogonal type~I and type~II MOPs as an equivalent matrix problem, leading to an IEP. Section~\ref{SSec3} introduces an updating procedure together with a novel scaling strategy designed to enhance numerical stability. In Section~\ref{SSec4}, we develop a downdating approach based on deflating an eigenvalue from the recurrence matrix. Section~\ref{SSec5} presents numerical experiments that demonstrate the effectiveness, stability, and accuracy of the proposed methods. Finally, Section \ref{sec:conclusion} gives a summary and discusses possible directions for future research.
\section{Problem formulation}\label{SSec2}
This section provides an overview of the step-line recurrences for both type I and type II MOPs, along with their biorthogonality property. We then set up the problem of generating type I and type II MOPs as an IEP.

   \subsection{Step-line recurrences}\label{Sec2}
MOPs admit different types of recurrence relations depending on the path chosen for the multi-indices in $\mathbb{N}^r$ \cite{CouVanass05,Van-nearestneighbor,FilHanVanass15,Vana24,MR2542683}. Here, we focus on the step-line indexing and consider the associated type I and type II MOPs
\[
A_{n,j}(x) = A_{\bm n,j}(x), \quad P_n(x) = P_{\bm n}(x), \qquad n = kr + \ell, \quad 
\bm n = (\underbrace{k+1,\dots,k+1}_{\ell},k,\dots,k),
\]
where $r$ denotes the number of measures. It is known \cite{CouVanass05,MR1985676,Faghih2025KrylovCore} that the sequence $\{P_n\}_{n=0}^{N-1}$ satisfies an $(r+2)$-term recurrence relation of the form
\[
xP_n(x) = a_{n+1,r+1} P_{n+1}(x) + \sum_{j=0}^{r} a_{n,j} P_{n-j}(x), \qquad 0 \le n \le N-1,
\]
with $P_0$ being a polynomial of degree zero (i.e., a nonzero constant) and $P_{-j}\equiv 0$ for $1\le j\le r$.

We assume a discrete reference measure $\omega=\sum_{i=1}^N \delta_{z_i}$ and write $d\omega_j(x)=\alpha_j(x)d\omega(x)$. The type I functions associated with the step-line indexing are defined by \cite{CouVanass05,Faghih2025KrylovCore,MR2542683}
\[
Q_{n}(x) = \sum_{j=1}^r A_{ n,j}(x)\alpha_j(x),
\]
and, by construction from the type I vector polynomial orthogonality relations \eqref{type1}, they satisfy the discrete orthogonality conditions
\[
\sum_{i=1}^N Q_{n}(z_i) z_i^k =
0, \quad 0\le k \le  n-2.
\]
Along the step-line, the type I functions admit the following recurrence relation involving the same coefficients
\[
xQ_n(x)=a_{n-1,r+1}Q_{n-1}(x)+\sum_{j=0}^r a_{n+j-1,j}Q_{n+j}(x),\quad \quad 1\le n \le N-r,
\]
    with initial conditions $Q_{0} = 0$ and $Q_{1},Q_{2},\ldots,Q_{r}$.

For simplicity of presentation and implementation, we focus on the case $r=2$ throughout the remainder of the paper. Extensions to the general case $r>2$ are possible, but are not pursued here to avoid overloading the notation. In this setting, the recurrences reduce to the four-term relations
\[
xP_n(x)=a_{n+1} P_{n+1}(x)+b_nP_n(x)+c_nP_{n-1}(x)+d_nP_{n-2}(x),
\]
\[
xQ_n(x)=a_{n-1} Q_{n-1}(x)+b_{n-1}Q_n(x)+c_nQ_{n+1}(x)+d_{n+1}Q_{n+2}(x).
\]
The coefficients $\{a_{n+1},b_n,c_n,d_n\}$ define a $4$-banded upper Hessenberg matrix 
  \begin{equation}\label{recurrencematrix}
   H_{N}=
    \begin{bmatrix}
    b_0 & c_1 & d_2 & 0 & 0&0&\cdots & 0 \\
    a_{1} & b_1 & c_2 & d_3 & 0 &0& \cdots & 0 \\
    0 & a_{2} & b_2 & c_3 & d_4 &0& \cdots & 0 \\
    0 & 0 & a_{3}& b_3 & c_4 & d_5 & \cdots & 0 \\
    \vdots &  & \ddots & \ddots & \ddots & \ddots & \ddots &  \\
    0 & \cdots & 0 & 0 & a_{N-3} & b_{N-3} & c_{N-2} & d_{N-1} \\
    0 & \cdots & 0 & 0 & 0 & a_{N-2} & b_{N-2} & c_{N-1} \\
    0 & \cdots & 0 & 0 & 0 & 0 & a_{N-1} & b_{N-1}
    \end{bmatrix}.
    \end{equation}
In this work, we do not impose the monicness condition on the type II MOPs, as relaxing this constraint provides additional flexibility that can be exploited to enhance numerical stability. As a result, $P_0(x)$ is not constrained to be one, and the subdiagonal entries of the Hessenberg matrix $H_N$ generally differ from one, in contrast to the convention used in other works \cite{CouVanass05,Faghih2025KrylovCore,Vana24}.

Consequently, the step-line recurrence relation of type II MOPs admits the vector-matrix formulation
\begin{equation}\label{steplinerecurrence}
        x \left[ P_0(x)  \; \ldots \; P_{N-1}(x) \right] 
        + \left[ 0 \; \ldots \; P_N(x) \right]= \left[ P_0(x) \; \ldots \; P_{N-1}(x) \right] H_N .
\end{equation}
Under the uniqueness conditions stated in the introduction, the zeros of $P_N(x)$ coincide with the nodes $\{z_i\}_{i=1}^N$ corresponding to the measures $\omega_j$, cf.~\eqref{MEASURE}. Evaluating the step-line recurrence \eqref{steplinerecurrence} at these nodes then leads to the following eigenvalue relation
\begin{equation}\label{eq:VZ}
ZV_{N} = V_{N} H_N,
\end{equation}
    with 
    \begin{align*}
        \begin{array}{cc}
            V_{N} =
            \begin{bmatrix}
                P_0(z_1) & P_1(z_1) & \cdots & P_{N-1}(z_1) \\
                P_0(z_2) & P_1(z_2) & \cdots & P_{N-1}(z_2) \\
                \vdots   & \vdots   &        & \vdots       \\
                P_0(z_N) & P_1(z_N) & \cdots & P_{N-1}(z_N)
            \end{bmatrix},\; \;
            Z = 
            \begin{bmatrix}
                z_1 & & &\\
                 & z_2&& \\
                    & &\ddots &\\
                   & & & z_N
            \end{bmatrix}.
        \end{array}
    \end{align*}
In particular, each zero $z_i$ of $P_N$ is an eigenvalue of $H_N$ with left eigenvector $\big[P_0(z_i)\ \cdots\ P_{N-1}(z_i)\big]$.
This spectral characterization underlies the inverse eigenvalue formulation developed next.

    \subsection{Biorthogonality and inverse eigenvalue problem}
The problem of constructing biorthogonal type I functions $\{Q_{n}\}_{n=1}^{N}=\{\alpha_{1}A_{n,1}+\alpha_{2}A_{n,2}\}_{n=1}^{N}$ and type II polynomials $\{P_{n}\}_{n=0}^{N-1}$, orthogonal with respect to a given discrete inner product, can be recast as a matrix problem. In this article, we base all techniques on the corresponding matrix formulation: we examine how up- and downdating the set of nodes in the inner products translates into up- and downdating the matrix of recurrences defining the recurrence relations for type~I functions and type~II MOPs.

The discrete inner product considered throughout the paper is of the form
\begin{equation}\label{INNER}
\langle P_{n},Q_{m}\rangle _N=\sum_{i=1}^N P_n(z_i)  Q_{m}(z_i)  =\sum_{i=1}^N P_n(z_i)  \big(\alpha_{1,i}A_{m,1}(z_i)+\alpha_{2,i}A_{m,2}(z_{i})\big),
\end{equation}
where $\{z_i\}_{i=1}^{N} \in \mathbb{R}$ are the distinct nodes, and $\{\alpha_{1,i}\}_{i=1}^{N}$ and $\{\alpha_{2,i}\}_{i=1}^{N}$ are the corresponding real positive weights for the discrete positive measures $\omega_1$ and $\omega_2$, respectively. The type~I functions and type~II polynomials satisfy the following biorthogonality relation with respect to $\langle \cdot, \cdot \rangle_N$ \cite[Chapter~23]{MR2542683}
\begin{equation}\label{BIORTHOGONALITY}
\langle P_n, Q_m \rangle_N =
\begin{cases}
0, & m \le n,\\
0, & n \le m-2,\\
1, & m = n+1.
\end{cases}
\end{equation}
To rewrite \eqref{BIORTHOGONALITY} in matrix form, we introduce the $N \times N$ matrix
\begin{equation*}
W_{N} =
\begin{bmatrix}
Q_1(z_1) & Q_2(z_1) & \cdots & Q_N(z_1) \\
Q_1(z_2) & Q_2(z_2) & \cdots & Q_N(z_2) \\
\vdots   & \vdots   &        & \vdots   \\
Q_1(z_N) & Q_2(z_N) & \cdots & Q_N(z_N)
\end{bmatrix}.
\end{equation*}
Then the biorthogonality \eqref{BIORTHOGONALITY} can be expressed compactly as $
W_{N}^\top V_{N} = I_N$,
where $I_N$ denotes the $N\times N$ identity matrix. Combining the biorthogonality with \eqref{eq:VZ} shows that
\[
Z W_{N} = W_{N} H_N^\top, \qquad  W_{N}^\top Z V_{N}=H_N.
\]
As a result, the $i$-th column of $W_{N}^\top$ provides a right eigenvector of $H_N$ corresponding to the eigenvalue $z_i$.

Consider the following IEP (Problem \ref{Prob:IEP}): given a set of eigenvalues, the first two entries of the right eigenvectors and the first entry of the left eigenvectors, 
construct a structured matrix (banded upper Hessenberg) that possesses the prescribed spectral and eigenvector properties.
\begin{problem}[IEP]\label{Prob:IEP}
Given a diagonal matrix $Z =
diag(\{z_{i}\}_{i=1}^{N}) \in \mathbb{R}^{N\times N}$ with distinct nodes, and vectors $\bm{w}_{1},\bm{w}_{2}, \bm{v}_{1} \in \mathbb{R}^{N}$ satisfying
\begin{equation*}
    \bm{w}_1^{\top} \bm{v}_1 = 1, \quad \bm{w}_2^{\top} \bm{v}_1 = 0, \quad \text{and} \quad \bm{w}_2^{\top} Z \bm{v}_1 = 1,    
\end{equation*}
construct a Hessenberg matrix $H_{N}\in \mathbb{R}^{N \times N}$, as in \eqref{recurrencematrix}, and a pair of matrices $ W_{N}, V_{N} \in \mathbb{R}^{N \times N} $ such that
\begin{enumerate}[label=(A\arabic*)]
	\item\label{cond:1} The first two columns of $W_{N}$ equal $\bm{w}_1$ and $\bm{w}_2$ and the first column of $V_{N}$ is $\bm{v}_1$:
    \[
	\begin{bmatrix}
		\mid & \mid \\
		\bm{w}_1 & \bm{w}_2 \\
		\mid & \mid
	\end{bmatrix}
	= W_{N} \begin{bmatrix}
		\mid & \mid \\
		\bm{e}_1 & \bm{e}_2 \\
		\mid & \mid
	\end{bmatrix},\ \quad \text{and} \quad
	\begin{bmatrix}
		\mid \\
		\bm{v}_1 \\
		\mid 
	\end{bmatrix}
	= V_{N} \begin{bmatrix}
		\mid \\
		\bm{e}_1 \\
		\mid 
	\end{bmatrix},\]\\
	\item\label{cond:2} $W_{N}^{\top}V_{N} = I_{N}$ (biorthogonality),
	\item \label{cond:3} $W_{N}^{\top} Z V_{N}=H_{N}$.
\end{enumerate}
\end{problem}
Given a discrete inner product as in \eqref{INNER}, the problem of constructing the associated 
biorthogonal type I functions and type II polynomials amounts to solving Problem~\ref{Prob:IEP}, where the eigenvalues of $H_{N}$ correspond to the set of nodes \( \{ z_i \}_{i=1}^N \), 
while \( \bm{v}_1 \) contains the first components of the left eigenvectors, 
and \( \bm{w}_1 \) and \( \bm{w}_2 \) contain the first and second components 
of the right eigenvectors, respectively. The vectors $\bm{w}_1, \bm{w}_2, \bm{v}_1$ are obtained by using the following proposition, which is a direct application of \cite[Theorem~2]{Vana24}.
\begin{proposition}\label{prop:initialvectors}
    \cite[Proposition 3.7]{Faghih2025KrylovCore} Assume we have an AT system with the positive discrete measures $\omega_{1}$ and $\omega_{2}$ as in \eqref{MEASURE}.
    Then, the starting vectors $\bm{w}_1,\bm{w}_2$ and $\bm{v}_1$ for Problem \ref{Prob:IEP} are given as 
    \begin{align*}
        \bm{w}_{1} &= a\,\bm{\alpha}_1, \quad \quad \quad \quad \quad \quad  \quad \bm{v}_1 = \begin{bmatrix}
            P_{0} & P_{0} & \ldots &P_{0}
        \end{bmatrix}^\top,\\
        \bm{w}_2 &= c\,\bm{\alpha}_1+ b\,\bm{\alpha}_2,
    \end{align*} where
    \begin{equation*}
        a= \dfrac{1}{P_{0}\sum_{i=1}^{N}\alpha_{1,i}},\quad b=\dfrac{1}{\sum_{i=1}^N \mathcal{C}_{i}\, \alpha_{2,i}}, \quad c= -a\, b\, \sum_{i=1}^{N}\alpha_{2,i},
    \end{equation*}
    along with $\mathcal{C}_{i}=z_i-\dfrac{\sum_{j=1}^N z_j\,\alpha_{1,j}}{\sum_{j=1}^N\alpha_{1,j}}$.
\end{proposition}
Using the characterization in Proposition~\ref{prop:initialvectors}, condition~\ref{cond:1} can be rewritten as
\[
	\mathcal{X}\mathcal{R}
	= W_{N} \begin{bmatrix}
		\mid & \mid \\
		\bm{e}_1 & \bm{e}_2 \\
		\mid & \mid
	\end{bmatrix},\ \quad \text{and} \quad
	\begin{bmatrix}
		\mid \\
		\bm{v}_1 \\
		\mid 
	\end{bmatrix}
	= V_{N} \begin{bmatrix}
		\mid \\
		\bm{e}_1 \\
		\mid 
	\end{bmatrix},\]
where $\mathcal{R}=\begin{bmatrix}
    a & c \\
    0 & b
\end{bmatrix}$, and $\mathcal{X}=\begin{bmatrix}
		\mid & \mid \\
		\bm{\alpha}_1 & \bm{\alpha}_2 \\
		\mid & \mid
	\end{bmatrix} $ 
    collects the weight vectors. This representation will be used throughout the remainder of the paper, since it provides a convenient framework for tracking the transformations of both the weight vectors and the coefficients in \(\mathcal R\).

We note again that, unlike the formulation in \cite{Faghih2025KrylovCore}, we do not assume $P_0$ to be monic; instead, $P_0$ may be any nonzero constant. This led to a minor modification of Proposition~\ref{prop:initialvectors} compared to Proposition~3.7 in \cite{Faghih2025KrylovCore}, where $P_0$ appears in the denominator of $a$ instead of $1$.

Faghih et al.~\cite{Faghih2025KrylovCore} proposed a Lanczos-type algorithm and a method based on core transformations to solve Problem \ref{Prob:IEP} efficiently. In this paper, we propose updating and downdating procedures which are more memory efficient than the Lanczos and core transformation methods, since they do not require storing the entire basis. Furthermore, the updating and downdating procedures are more flexible, as they allow the reuse of existing recurrence relations to construct new recurrences under modest modifications of the inner product, such as adding or deleting nodes. This flexibility is not possible with the other approaches, which would require a costly restart from scratch.

\section{Updating the IEP}\label{SSec3}
In what follows, we first formulate the updating problem for IEP \ref{Prob:IEP} 
and subsequently present the procedure for efficiently constructing the updated recurrence matrix.
\begin{problem}[Update IEP]\label{Update:IEP}
A solution to an IEP \ref{Prob:IEP} of size \( N \) 
is assumed to be available, consisting of a Hessenberg matrix 
\( H_N \in \mathbb{R}^{N \times N} \) and a pair of biorthonormal matrices 
\( V_{N}, W_{N} \in \mathbb{R}^{N \times N} \) satisfying \ref{cond:1}, \ref{cond:2}, and \ref{cond:3}. Given additional nonzero weights \( \alpha_{1,N+1}, \alpha_{2,N+1} \) and a node 
\( z_{N+1} \notin \{ z_i \}_{i=1}^N \), compute the upper Hessenberg matrix $\widetilde{H}_{N+1}$ and the biorthonormal pair 
\( \widetilde{V}_{N+1}, \widetilde{W}_{N+1} \) satisfying
\begin{enumerate}[label=(B\arabic*)]
	\item\label{cond:11}
    \[\mathcal{X}_{\mathrm{up}}\mathcal{R}_{\mathrm{up}}
	= \widetilde{W}_{N+1} \begin{bmatrix}
		\mid & \mid \\
		\bm{e}_1 & \bm{e}_2 \\
		\mid & \mid
	\end{bmatrix},\ \quad \text{and} \quad
	\begin{bmatrix}
		\mid \\
		\bm{\tilde{v}}_1 \\
		\mid 
	\end{bmatrix}
	= \widetilde{V}_{N+1} \begin{bmatrix}
		\mid \\
		\bm{e}_1 \\
		\mid 
	\end{bmatrix},\]\\
	\item\label{cond:22} $\widetilde{W}^{\top}_{N+1}\widetilde{V}_{N+1} = I_{N+1}$,
	\item \label{cond:33} $\widetilde{W}^{\top}_{N+1} \underbrace{
\begin{bmatrix}
Z &  \\
 & z_{N+1}
\end{bmatrix}}_{= \widetilde{Z}}\widetilde{V}_{N+1}=\widetilde{H}_{N+1}$.
\end{enumerate}
Here, the subscript ``up'' refers to the updated quantities
\[\mathcal{X}_{\mathrm{up}}
= 
\begin{bmatrix}
    \mid & \mid \\
    \bm{\tilde{\alpha}}_1 & \bm{\tilde{\alpha}}_2 \\
    \mid & \mid 
\end{bmatrix}, \quad \mathcal{R}_{\mathrm{up}}=\begin{bmatrix}
    \tilde{a} & \tilde{c} \\
    0 & \tilde{b}
\end{bmatrix},\]
and the coefficients \(\tilde a\), \(\tilde b\), and \(\tilde c\) are computed according to Proposition~\ref{prop:initialvectors} using the updated weight vectors obtained by appending the new weights \(\alpha_{1,N+1}\) and \(\alpha_{2,N+1}\) to the original vectors.
That is,
\[
\boldsymbol{\tilde{\alpha}}_1 =
\begin{bmatrix}
\alpha_{1,1} & \alpha_{1,2} & \ldots & \alpha_{1,N} & \alpha_{1,N+1}
\end{bmatrix}^{\top}, 
\qquad
\boldsymbol{\tilde{\alpha}}_2 =
\begin{bmatrix}
\alpha_{2,1} & \alpha_{2,2} & \ldots & \alpha_{2,N} & \alpha_{2,N+1}
\end{bmatrix}^{\top}.
\]
\end{problem}

We now describe an updating procedure for Problem~\ref{Update:IEP}, 
demonstrating how eliminators can be used to restore and maintain the 
banded upper Hessenberg structure of a matrix, both in its upper 
and lower subdiagonal parts. The updating process begins with a known structured matrix 
\( H_{N} \in \mathbb{R}^{N \times N} \) solving an IEP of size \( N \). A new matrix 
$\widetilde{H}_{N+1} \in \mathbb{R}^{(N+1) \times (N+1)} $
is then constructed to solve a related problem of dimension \( N+1 \), 
obtained by introducing an additional node and its associated weights into the underlying discrete inner product. In general, the updating algorithm proceeds through three main stages
\begin{itemize}
    \item Extend all matrices associated with the solution 
    in \( \mathbb{R}^{N \times N} \) to higher-dimensional matrices 
    in \( \mathbb{R}^{(N+1) \times (N+1)} \), and denote the extended versions 
    by a tilde symbol.
    
    \item Impose biorthogonality with respect to the newly introduced weight vectors; 
    this step perturbs the structure of \( \widetilde{H}_{N+1} \in \mathbb{R}^{(N+1) \times (N+1)} \).
    
    \item Restore the prescribed structural form of 
    \( \widetilde{H}_{N+1} \).
\end{itemize}

Breakdowns may occur in the updating procedure discussed in this section; 
however, for simplicity, we assume that no such breakdowns take place.

We employ eliminators for the updating procedure, which are essentially \( 2 \times 2 \) 
triangular matrices. Let \( \mathfrak{L}_{N+1}\subset \mathbb{R}^{(N+1) \times (N+1)}\) denote the class of lower triangular 
eliminators and \( \mathfrak{R}_{N+1} \subset \mathbb{R}^{(N+1) \times (N+1)} \) the class of upper triangular eliminators. 
The classes \( \mathfrak{L}_{N+1} \) 
and \( \mathfrak{R}_{N+1}  \) 
consist, respectively, of matrices of the form
\[
L_{i} =
\begin{bmatrix}
I_{i-1} &  & & \\
& 1 & & \\
 & & I_{N-i} \\
 &\ell_{i} &&1
\end{bmatrix},
\qquad
R_{i} =
\begin{bmatrix}
I_{i+1} &  &  &  \\
 & 1 &  & r_{i+2} \\
 &  & I_{N-i-2} &  \\
 &  &  &1
\end{bmatrix},
\]
with parameters \( \ell_{i}, r_{i+2} \in \mathbb{R} \), appearing in the entries $(N+1,i)$ and $(i+2,N+1)$, respectively.
For further reading on how eliminators can be used to solve IEPs, we refer to \cite{VB21}.

The biorthonormal matrices \( V_{N}, W_{N} \in \mathbb{R}^{N \times N} \) and the matrix $H_{N}$ are embedded in larger
matrices while maintaining the biorthogonality property such that $\widetilde{W}^{\top}_{N+1}\widetilde{Z}
\widetilde{V}_{N+1}=\widetilde{H}_{N+1}$,
\begin{equation*}
	\widetilde{V}_{N+1}=\begin{bmatrix}
		V_{N}& \\
		& 1
	\end{bmatrix}, \quad 	\widetilde{W}_{N+1}=\begin{bmatrix}
	W_{N}& \\
	& 1
\end{bmatrix},\quad \text{and} \quad	{\widetilde{H}_{N+1}}=\begin{bmatrix}
H_{N}& \\
& z_{N+1}
\end{bmatrix}.
\end{equation*}
The first column of \(\widetilde{V}_{N+1}\) and the first two columns of \(\widetilde{W}_{N+1}\) do not yet satisfy condition~\ref{cond:11}. From Proposition \ref{prop:initialvectors}, the first and second columns of $\widetilde{W}_{N+1}$ can be written as
\begin{equation*} 
\mathcal{W}_{\mathrm{pre}}=\begin{bmatrix}
    \mid & \mid \\
    \bm{w}_1 & \bm{w}_2 \\
    \mid & \mid \\
    0 & 0
\end{bmatrix}
=
\begin{bmatrix}
    \mid & \mid \\
    a\,\bm{\alpha}_1 & c\,\bm{\alpha}_1 + b\,\bm{\alpha}_2 \\
    \mid & \mid \\
    0 & 0
\end{bmatrix}=\mathcal{X}_{\mathrm{pre}}\,
\mathcal{R}_{\mathrm{pre}},
\end{equation*}
with $\mathcal{X}_{\mathrm{pre}}
= 
\begin{bmatrix}
    \mid & \mid \\
    \bm{\alpha}_1 & \bm{\alpha}_2 \\
    \mid & \mid \\
    0 & 0
\end{bmatrix}$, and $\mathcal{R}_{\mathrm{pre}}=\begin{bmatrix}
    a & c \\
    0 & b
\end{bmatrix}$, which are extended versions of $\mathcal{X}$ and $\mathcal{R}$.

A suitable transformation is applied to $\widetilde{W}_{N+1}$ to alter the first two columns  so that
they equal $\mathcal{W}_{\mathrm{up}}=\mathcal{X}_{\mathrm{up}}\,
\mathcal{R}_{\mathrm{up}}$.
\begin{lemma}
    A transformation $T \in \mathbb{R}^{(N+1)\times(N+1)}$ exists that transforms the matrix $\widetilde{W}_{N+1}$, i.e.,
    \begin{equation*}
   \widetilde{W}_{N+1} T=\widetilde{W}_{N+1}^{[1]}=
\begin{bmatrix}
     & \mid && \mid & 0 \\
    \mathcal{W}_{\mathrm{up}} & \bm{w}_3 & \cdots & \bm{w}_N & \vdots \\
     & \mid && \mid & 0 \\
     & 0 && 0 & 1
\end{bmatrix}
\end{equation*}
\end{lemma}
\begin{proof}
    Defining $\tilde{\bm{r}} = 
\begin{bmatrix}
\alpha_{1,N+1} & \alpha_{2,N+1} & 0 & \cdots & 0
\end{bmatrix} \in \mathbb{R}^{1 \times N}$, we consider the transformation $T$ as
\begin{equation*}
T=\begin{bmatrix}
    \mathcal{R}_{\mathrm{pre}}^{-1} &  \bold{0}\\
    \bold{0} & I_{N-1} 
\end{bmatrix}
\begin{bmatrix}
I_N & 0 \\
\tilde{\bm{r}} & 1
\end{bmatrix}\begin{bmatrix}
    \mathcal{R}_{\mathrm{up}}&  \bold{0}\\
    \bold{0} & I_{N-1} 
\end{bmatrix}.
\end{equation*}
Applying this transformation to $\widetilde{W}_{N+1}$, we get
 \begin{eqnarray*}
   \widetilde{W}_{N+1} T&=&
\begin{bmatrix}
     & \mid && \mid & 0 \\
    \mathcal{W}_{\mathrm{pre}} & \bm{w}_3 & \cdots & \bm{w}_N & \vdots \\
     & \mid && \mid & 0 \\
     & 0 && 0 & 1
\end{bmatrix}\begin{bmatrix}
    \mathcal{R}_{\mathrm{pre}}^{-1} &  \bold{0}\\
    \bold{0} & I_{N-1} 
\end{bmatrix}
\begin{bmatrix}
I_N & 0 \\
\tilde{\bm{r}} & 1
\end{bmatrix}\begin{bmatrix}
    \mathcal{R}_{\mathrm{up}}&  \bold{0}\\
    \bold{0} & I_{N-1} 
\end{bmatrix}\\
&=&\begin{bmatrix}
     & \mid && \mid & 0 \\
    \mathcal{X}_{\mathrm{pre}} & \bm{w}_3 & \cdots & \bm{w}_N & \vdots \\
     & \mid && \mid & 0 \\
     & 0 && 0 & 1
\end{bmatrix}\begin{bmatrix}
I_N & 0 \\
\tilde{\bm{r}} & 1
\end{bmatrix}\begin{bmatrix}
    \mathcal{R}_{\mathrm{up}}&  \bold{0}\\
    \bold{0} & I_{N-1} 
\end{bmatrix}\\
&=&\begin{bmatrix}
     & \mid && \mid & 0 \\
    \mathcal{X}_{\mathrm{up}} & \bm{w}_3 & \cdots & \bm{w}_N & \vdots \\
     & \mid && \mid & 0 \\
     & 0 && 0 & 1
\end{bmatrix}\begin{bmatrix}
    \mathcal{R}_{\mathrm{up}}&  \bold{0}\\
    \bold{0} & I_{N-1} 
\end{bmatrix}=\widetilde{W}_{N+1}^{[1]}.
\end{eqnarray*}
Hence, the transformation \(T\) updates \(\widetilde{W}_{N+1}\) as required, completing the proof.
\end{proof}
We also set $\widetilde{V}_{N+1}^{[1]} = \widetilde{V}_{N+1} T^{-\top}$ to preserve the biorthogonality condition~\ref{cond:22}, $ \widetilde{W}_{N+1}^{[1]\top} \widetilde{V}_{N+1}^{[1]} = I_{N+1}$. Recalling the transformation $T$, we obtain
 \begin{eqnarray*}
  \widetilde{V}_{N+1}^{[1]}= \widetilde{V}_{N+1} T^{-\top}&=&
\widetilde{V}_{N+1}\begin{bmatrix}
    \mathcal{R}_{\mathrm{pre}}^{-1} &  \bold{0}\\
    \bold{0} & I_{N-1} 
\end{bmatrix}^{-\top}
\begin{bmatrix}
I_N & 0 \\
\tilde{\bm{r}} & 1
\end{bmatrix}^{-\top}\begin{bmatrix}
    \mathcal{R}_{\mathrm{up}}&  \bold{0}\\
    \bold{0} & I_{N-1} 
\end{bmatrix}^{-\top}\\
&=&\widetilde{V}_{N+1}\begin{bmatrix}
    \mathcal{R}_{\mathrm{pre}}^{\top} &  \bold{0}\\
    \bold{0} & I_{N-1} 
\end{bmatrix}
\begin{bmatrix}
I_N & -\tilde{\bm{r}}^{\top} \\
0 & 1
\end{bmatrix}\begin{bmatrix}
    \mathcal{R}_{\mathrm{up}}^{-\top}&  \bold{0}\\
    \bold{0} & I_{N-1} 
\end{bmatrix}.
\end{eqnarray*}
From this representation it follows that $
e_{N+1}^{\top}\widetilde{V}_{N+1}^{[1]} = e_{N+1}^{\top}$. An additional transformation can be applied to fix the first column of $\widetilde{V}_{N+1}^{[1]}$. Specifically, we choose $L_{1} \in \mathfrak{L}_{N+1}$ with $l_{1}=P_{0}$, which leads to the updated matrices
\[
\widetilde{V}_{N+1}^{[2]} = \widetilde{V}_{N+1}^{[1]} L_{1}, 
\qquad
\widetilde{W}_{N+1}^{[2]} = \widetilde{W}_{N+1}^{[1]} L_{1}^{-\top}.
\]

It should be noted that the transformation $L_{1}^{-\top}$ does not change the first two columns of $\widetilde{W}_{N+1}^{[1]}$. Consequently, the first two columns of $\widetilde{W}_{N+1}^{[2]}$ and the first column of $\widetilde{V}_{N+1}^{[2]}$ attain the required form, and hence condition~\ref{cond:11} is satisfied.

The change of basis implies that the relevant recurrence matrix is now defined as
\begin{equation*}
    \widetilde{H}_{N+1}^{[1]} =L_{1}^{-1} T^{\top}\widetilde{H}_{N+1} T^{-\top} L_{1},
    \label{eq:transformed_H}
\end{equation*}
which is no longer a Hessenberg matrix, and whose structure is shown as follows
 
\[
\widetilde{H}_{N+1}^{[1]} =
\left[
\begin{array}{cccccc|c}
\times & \times & \times &        &        &        & \otimes \\
\times & \times & \times & \times &        &        & \otimes \\
       & \times & \times & \times & \times &        & \vdots \\
       &        & \times & \times & \ddots & \ddots & \otimes \\
       &        &        & \times & \ddots & \times & \times \\
       &        &        &        & \ddots & \times & \times \\ \hline
\otimes & \otimes & \otimes &  \cdots & \otimes &\times & \times
\end{array}
\right].
\]
Throughout the paper, \(\times\) denotes a generic nonzero entry, whereas \(\otimes\) denotes a distinguished nonzero entry that is intended to be removed. To restore the $4$-banded Hessenberg structure, eliminators are constructed to eliminate the nonzero elements in the last row and last column, without altering the first column of $\widetilde{V}_{N+1}^{[2]} $, and the first two columns of $\widetilde{W}_{N+1}^{[2]} $. 

We remove $N - 1$ bulges from the last row of $\widetilde{H}_{N+1}^{[1]}$ and $N - 2$ from the last column to restore the proper four-banded Hessenberg form. In total, $2N - 3$ eliminations need to be performed. We start with the first element of the last row, then alternately remove one element from the last row and one from the last column, beginning with the second element of the last row, until all nonzero entries are eliminated. To eliminate the $s$-th element, for $s = 1, 2, \ldots, 2N - 3$, the corresponding eliminator is denoted by $P_{s}$. If the entry to be eliminated is located in the last row and the $j$-th column, we define
\[
P_{s} = L_{j+1} \in \mathfrak{L}_{N+1}, \qquad j = 1, \ldots, N - 1,
\]
and apply $P_{s}$ from the left, yielding
\[
\widetilde{H}_{N+1}^{[s+1]} = P_{s} \, \widetilde{H}_{N+1}^{[s]} \, P_{s}^{-1}.
\] 
The corresponding biorthogonal bases are updated as
\[
\widetilde{V}_{N+1}^{[s+2]} = \widetilde{V}_{N+1}^{[s+1]} P_{s}^{-1}, \quad 
\widetilde{W}_{N+1}^{[s+2]}= \widetilde{W}_{N+1}^{[s+1]} P_{s}^{\top}.
\] 
Conversely, if the entry to be eliminated is located in the last column and the $i$-th row, we define
\[
P_{s} = R_{i} \in \mathfrak{R}_{N+1}, \qquad i = 1, \ldots, N - 2,
\]
and apply $P_{s}$ from the right, so that
\[
\widetilde{H}_{N+1}^{[s+1]} = P_{s}^{-1} \, \widetilde{H}_{N+1}^{[s]} \, P_{s}.
\] 
The corresponding biorthogonal bases are then updated according to
\[
\widetilde{V}_{N+1}^{[s+2]} = \widetilde{V}_{N+1}^{[s+1]} P_{s}, \quad
\widetilde{W}_{N+1}^{[s+2]}= \widetilde{W}_{N+1}^{[s+1]} P_{s}^{-\top}.
\] 
The parameters used in the eliminators are given by
\[
\begin{cases}
l_{j+1} = -\dfrac{h^{[s]}_{N+1,j}}{h^{[s]}_{j+1,j}}, \\[1.2em]
r_{i+2} = -\dfrac{h^{[s]}_{i,N+1}}{h^{[s]}_{i,i+2}},
\end{cases}
\]
where $h^{[s]}_{i,j}$ denotes the $(i,j)$-th element of the matrix $\widetilde{H}_{N+1}^{[s]}$. Figure \ref{fig:elimination} provides an overview of the targeted entries for elimination and the pivots employed in the procedure.
\begin{SCfigure}[1.2]

\centering

\begin{tikzpicture}[baseline=(m.center), every node/.style={font=\scriptsize}]
  \matrix (m) [matrix of math nodes,
               left delimiter={[}, right delimiter={]},
               nodes in empty cells,
               nodes={minimum width=6mm, minimum height=5.5mm, anchor=center},
               row sep=1.2mm, column sep=1.2mm]{
    \times & \times &        &        &        &        &        \\
           & \times & \times &        &        &        &        \\
           &        & \times & \times &        &        &        \\
           &        &        & \times & \ddots & \ddots & \vdots \\
           &        &        &        & \ddots & \times & \times \\
           &        &        &        & \ddots & \times & \times \\
           &        &        &        & \cdots & \times & \times \\
  };

  \draw[line width=0.6pt]
    ($(m-1-6.north east)+(0.3mm,0)$)
    --
    ($(m-7-6.south east)+(0.3mm,0)$);

  \draw[line width=0.6pt]
    ($(m-6-1.south west)+(0,-0.3mm)$)
    --
    ($(m-6-7.south east)+(0,-0.3mm)$);

  \foreach \j in {1,2,3,4,5,6} {
    \coordinate (r\j) at (m-7-\j.center);
  }

  \foreach \i in {1,2,3,4,5} {
    \coordinate (c\i) at (m-\i-7.center);
  }

  \foreach \n/\pos in {1/r1,2/r2,4/r3,6/r4}{
    \node[
      draw,
      red,
      line width=0.6pt,
      circle,
      minimum size=4.2mm,
      inner sep=0pt
    ] at (\pos) {\textcolor{red}{\scriptsize \n}};
  }

  \foreach \n/\pos in {3/c1,5/c2,7/c3}{
    \node[
      draw,
      red,
      line width=0.6pt,
      circle,
      minimum size=4.2mm,
      inner sep=0pt
    ] at (\pos) {\textcolor{red}{\scriptsize \n}};
  }

  \foreach \num/\pos in {
    1/(m-2-1.center),
    2/(m-3-2.center),
    4/(m-4-3.center),
    3/(m-1-3.center),
    5/(m-2-4.center),
    7/(m-3-5.center),
    6/(m-5-4.center)
  }{
    \node[
      draw,
      blue,
      line width=0.6pt,
      circle,
      minimum size=4.5mm,
      inner sep=0pt
    ] at \pos {\textcolor{blue}{\scriptsize \num}};
  }

\end{tikzpicture}
\hspace{6mm}
\caption{Overview of the elimination procedure. The red circles indicate the elements to be removed, and the numbering illustrates the order in which they are eliminated. The blue circles highlight the pivot elements used in the elimination process, with numbers matching the corresponding red circles to indicate which pivot is used for each element.}
\label{fig:elimination}

\end{SCfigure}
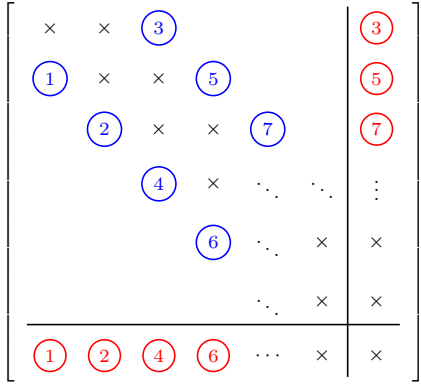

This procedure may, however, break down if the denominator of $l_{j}$ or $r_{i+2}$ becomes very small, leading to numerical instability. Under the assumption that no breakdown occurs, we repeat this process until the proper $4$-banded upper Hessenberg matrix $\widetilde{H}_{N+1}^{[2N-2]}$ is achieved. By defining \begin{equation*}
\widetilde{V}_{N+1}=\widetilde{V}_{N+1}^{[2N-1]},\quad 
 \widetilde{W}_{N+1}=\widetilde{W}_{N+1}^{[2N-1]}, \quad  \quad \widetilde{H}_{N+1}=\widetilde{H}_{N+1}^{[2N-2]},
\end{equation*}
the condition \ref{cond:33} is satisfied and the updated IEP \ref{Update:IEP} is solved. 
\subsection{Scaling strategy}\label{Scaling}
  The core transformations and the biorthogonal Lanczos procedure developed by Faghih et al. \cite{{Faghih2025KrylovCore}} produce a Hessenberg recurrence with unit subdiagonal entries. The unit entries on the subdiagonal follow from the monicness of the type~II MOPs; see \cite[Section 3]{Faghih2025KrylovCore}. In finite-precision arithmetic, the computation of the left and right bases can be highly ill-conditioned, so that the biorthogonality $W_{N}^{\top}V_{N} = I_N$ and the recurrence relation $W_{N}^{\top} Z V_{N} = H_{N}$ hold only approximately. Instead of the monicness condition, we consider non-monic polynomials, allowing the subdiagonal entries of the Hessenberg matrix to take values different from one. This flexibility decouples the subdiagonal from the monic constraint and provides additional degrees of freedom to balance the left and right bases.

We exploit this freedom by applying a diagonal scaling that enforces the following property for each column pair $(\bm{\tilde w}_i, \bm{\tilde v}_i)$
\[
\|\text{Scaled} ~\bm{\tilde w}_i\|_2 \approx \|\text{Scaled}~\bm{\tilde v}_i\|_2, \quad i = 1,2,\dots,N+1.
\]
This choice reduces the condition numbers of $\widetilde{W}_{N+1}$ and $\widetilde{V}_{N+1}$, mitigates numerical instabilities by keeping the two bases balanced in norm, and significantly improves the preservation of the biorthogonality $\widetilde{W}_{N+1}^\top \widetilde{V}_{N+1}= I_{N+1}$ and the recurrence relation $\widetilde{W}_{N+1}^\top \widetilde{Z} \widetilde{V}_{N+1}= \widetilde{H}_{N+1}$ in finite-precision computations.

To do this, we construct diagonal scaling matrices
\[
D_W = \operatorname{diag}\Big(\frac{1}{\beta_1}, \dots, \frac{1}{\beta_{N+1}}\Big), \qquad
D_V = \operatorname{diag}\big(\beta_1, \dots, \beta_{N+1}\big),
\]
with scaling factors
\[
\beta_i = \sqrt{\frac{\|\bm{\tilde w}_i\|_2}{\|\bm{\tilde v}_i\|_2}}, \quad i = 1, 2,\dots, N+1.
\]

The bases and Hessenberg matrix are then updated as
\[
\widetilde{V}^{S}_{N+1} = \widetilde{V}_{N+1} D_V, \quad
\widetilde{W}^{S}_{N+1} = \widetilde{W}_{N+1} D_W^\top, \quad
\widetilde{H}^{S}_{N+1} = D_W \widetilde{H}_{N+1} D_V.
\]

This transformation preserves the biorthogonality while ensuring that each pair of corresponding columns has approximately equal norm.
\section{Downdating the IEP}\label{SSec4}
Downdating of a Hessenberg
matrix is formulated in Problem \ref{Downdate:IEP} and essentially amounts to removing an eigenvalue from the Hessenberg’s spectrum.
\begin{problem}[Downdate IEP]\label{Downdate:IEP}
A solution to an IEP \ref{Prob:IEP} of size \( N \) 
is assumed to be available, consisting of a Hessenberg matrix 
\( H_N \in \mathbb{R}^{N \times N} \) and a pair of biorthonormal matrices 
\( V_{N}, W_{N} \in \mathbb{R}^{N \times N} \) satisfying \ref{cond:1}, \ref{cond:2}, and \ref{cond:3}. Given a node $\tilde{z} \in \{z_{i}\}_{i=1}^{N}$, assume, without loss of generality, $\tilde{z}=z_{j}$. Denote by $\widetilde{Z}=\operatorname{diag}(\{z_i\}_{i=1,i\neq j}^{N})$ the new matrix of nodes and by
\[
\boldsymbol{\tilde{\alpha}}_r =
\begin{bmatrix}
\alpha_{r,1} & \cdots & \alpha_{r,j-1} & \alpha_{r,j+1} & \cdots & \alpha_{r,N}
\end{bmatrix}^{\top},
\qquad r=1,2,
\]
the corresponding downdated weight vectors. Compute the upper Hessenberg matrix $\widetilde{H}_{N-1}$ and the biorthonormal pair 
\( \widetilde{V}_{N-1}, \widetilde{W}_{N-1} \) satisfying
\begin{enumerate}[label=(C\arabic*)]
	\item\label{cond:111}
    \[
	\mathcal{X}_{\mathrm{down}}\mathcal{R}_{\mathrm{down}}
	= \widetilde{W}_{N-1} \begin{bmatrix}
		\mid & \mid \\
		\bm{e}_1 & \bm{e}_2 \\
		\mid & \mid
	\end{bmatrix},\ \quad \text{and} \quad
	\begin{bmatrix}
		\mid \\
		\bm{\tilde{v}}_1 \\
		\mid 
	\end{bmatrix}
	= \widetilde{V}_{N-1} \begin{bmatrix}
		\mid \\
		\bm{e}_1 \\
		\mid 
	\end{bmatrix},\]\\
	\item\label{cond:222} $\widetilde{W}^{\top}_{N-1}\widetilde{V}_{N-1} = I_{N-1}$,
	\item \label{cond:333} $\widetilde{W}^{\top}_{N-1}  \widetilde{Z}\widetilde{V}_{N-1}=\widetilde{H}_{N-1}$.
\end{enumerate}
The matrices \(\mathcal X_{\mathrm{down}}\) and \(\mathcal R_{\mathrm{down}}\) are defined by 
\[\mathcal{X}_{\mathrm{down}}
= 
\begin{bmatrix}
    \mid & \mid \\
    \bm{\tilde{\alpha}}_1 & \bm{\tilde{\alpha}}_2 \\
    \mid & \mid 
\end{bmatrix}, \quad \mathcal{R}_{\mathrm{down}}=\begin{bmatrix}
    \tilde{a} & \tilde{c} \\
    0 & \tilde{b}
\end{bmatrix},\]
where the coefficients \(\tilde a\), \(\tilde b\), and \(\tilde c\) are computed according to Proposition~\ref{prop:initialvectors}.
\end{problem}

 To tackle this matrix problem, we rely on the eigenvalue decomposition of a related Hessenberg matrix. In what follows, we present an algorithm tailored to that setting. Throughout, we restrict our attention to unreduced Hessenberg matrices. This assumption guarantees that the recurrence will not break down, meaning a complete family of biorthogonal MOPs can be generated for the chosen inner product. The Hessenberg matrix under consideration is not assumed to be normal; however, since it is diagonalizable, it admits the eigen-decomposition $
H_{N}=  W_{N}^{\top}Z V_{N}$, since $W_{N}^{\top}=V_{N}^{-1}$.

The downdating step is based on isolating the prescribed eigenvalue
\(\tilde z\) by a similarity transformation that preserves the Hessenberg structure. Assume we have a solution to Problem~\ref{Prob:IEP}, i.e., $W_{N}^{\top} Z V_{N}=H_{N}$,
where the conditions \ref{cond:1}, \ref{cond:2}, and \ref{cond:3} are satisfied. The goal is to
construct a transformation $\mathcal{S}$ which separates the eigenvalue \(\tilde z\) from
the remaining spectrum, so that
\begin{eqnarray}\label{dd}
	\mathcal{S}^{-1} \, H_{N} \, \mathcal{S}=\mathcal{S}^{-1}\, W_{N}^{\top} \, Z \, V_{N}\, \mathcal{S}&=&\begin{bmatrix}
		1& \\
		& \widetilde{W}_{N-1}^{\top}
	\end{bmatrix} \begin{bmatrix}
	\tilde{z}& \\
	& \widetilde{Z}
\end{bmatrix}\begin{bmatrix}
1& \\
& \widetilde{V}_{N-1}
\end{bmatrix}\\
\nonumber &=& \begin{bmatrix}
\tilde{z}& \\
& \widetilde{H}_{N-1}
\end{bmatrix}.
\end{eqnarray}
The trailing principal submatrix \(\widetilde H_{N-1}\) then represents the desired solution, under the conditions \ref{cond:111}, \ref{cond:222}, and \ref{cond:333}. The matrix \(\mathcal S\) denotes the structured similarity transformation associated with the chosen downdating procedure. This imposes the following three constraints on the downdating procedure:
\begin{itemize}
  \item Equation~\eqref{dd} must have the correct block form (block-diagonal), so that \(\tilde{z}\) can be removed.
  \item The matrix \(\widetilde H_{N-1}\) must remain a banded upper-Hessenberg matrix with one subdiagonal and two superdiagonals.
  \item The first column of $\widetilde{V}_{N-1}$ and the first and second columns of $\widetilde{W}_{N-1}$ must coincide with the newly defined \( \bm{\tilde{v}}_1, \bm{\tilde{w}}_1, \bm{\tilde{w}}_2 \).
\end{itemize}
\subsection{Isolation of an eigenvalue}
The QR algorithm is one of the classical methods for computing the eigenvalues of dense matrices of moderate size. Each
iteration involves a shift, and the eigenvalue closest to the shift is
gradually driven toward the lower-right corner of the matrix, where it
can eventually be isolated and deflated. By repeating this procedure,
all eigenvalues may successively be extracted. A shift is called
\emph{perfect} when it is chosen to be an exact eigenvalue of the
matrix. In exact arithmetic, such a choice leads to immediate deflation in a single QR step. In particular, eigenvector-based constructions for perfect-shift QR steps were studied by Mastronardi and Van Dooren~\cite{MR3867618}.

In the framework considered here, we instead employ RL iterations with perfect shifts, as they preserve the biorthogonal structure; see Watkins~\cite{b333}. Our construction is inspired by the eigenvector-based
approach of Mastronardi and Van Dooren~\cite{MR3867618}. The procedure consists of two perfect-shift RL steps: the first is constructed from a left eigenvector associated with the chosen shift, while the second is constructed from a corresponding right eigenvector. Assume that a number $\tilde{z}$ is available that is sufficiently close to an eigenvalue of the matrix $H_N$ (which, in theory, should be exact in our setting). The eigenvector-based method begins by computing a normalized eigenvector $\bm{x}$ such that
\[
\|(H_N - \tilde{z} I_{N}) \bm{x}\|_2 \approx \epsilon_{mach} \, \|H_N - \tilde{z} I_{N}\|_2,
\]
i.e., the residual is on the order of machine precision.  

In our framework, the eigenvector $\bm{x}$ can be obtained directly, once the solution of Problem \ref{Prob:IEP} is available. If $\tilde{z} = z_j$, then the $j$-th row of $V_N$, 
$\bm{x}_{\ell e} = [\,P_0(\tilde{z}) \; P_1(\tilde{z}) \; \cdots \; P_{N-1}(\tilde{z})\,]$, 
serves as the left eigenvector of $H_N$ corresponding to $\tilde{z}$. 
Likewise, the $j$-th row of $W_N$, 
$\bm{x}_{ri} = [\,Q_1(\tilde{z}) \; Q_2(\tilde{z}) \; \cdots \; Q_{N}(\tilde{z})\,]$, 
is the right eigenvector associated with the same eigenvalue~$\tilde{z}$. Since the first perfect-shift RL step is constructed from the left
eigenvector, we initially work with $\bm{x}_{le}$. 
The accuracy of $\bm{x}_{le}$ can be enhanced by performing a few steps of iterative refinement. The resulting refined eigenvector $\bm{\dot{x}}_{le}$ is computed with particular emphasis on the accuracy of its trailing components. The refinement strategy and the corresponding error analysis are described by Mastronardi and Van Dooren~\cite[Section~5]{MR3867618}; see also~\cite{VB21}.

We begin with the first perfect RL shift. The shifted Hessenberg matrix $H_{N}-\tilde{z} I_{N}$ is singular, and since $H_{N}$ is unreduced Hessenberg, the matrix $H_{N}-\tilde{z} I_{N}$ has rank $N-1$. Hence the nullspace of $H_{N}-\tilde{z} I_{N}$ is one-dimensional, and there exists a normalized left eigenvector $\bm{\dot{x}}_{le}$ satisfying
\[
   \bm{\dot{x}}_{le} H_{N}=\tilde{z}\, \bm{\dot{x}}_{le},
    \qquad
    \|\bm{\dot{x}}_{le}\|_2=1.
\]
This eigenvector is unique up to a scale factor $\pm 1$.

The left eigenvector $\bm{\dot{x}}_{le}$ allows us to
eliminate the first upper diagonal of the shifted matrix
$H_N-\tilde{z}I_{N}$. However, in order to remove the second upper diagonal, an additional auxiliary vector is required. To construct this vector, we consider the upper-right submatrix of order $N-1$ of the shifted matrix $H_N-\tilde{z}I_N$, denoted by $\widehat{H}_{N-1}$,
\[
\widehat{H}_{N-1}
=
\begin{bmatrix}
h_{1,2} & h_{1,3} & 0 & \cdots & 0 \\
h_{2,2}-\tilde{z} & h_{2,3} & h_{2,4} & \ddots & \vdots \\
h_{3,2} & h_{3,3}-\tilde{z} & h_{3,4} & \ddots & 0 \\
0 & h_{4,3} & h_{4,4}-\tilde{z} & \ddots & 0 \\
\vdots & \ddots & \ddots & \ddots & h_{N-2,N} \\
0 & \cdots & h_{N-1,N-2} & h_{N-1,N-1}-\tilde{z} & h_{N-1,N}
\end{bmatrix}.
\]
Since $\widehat{H}_{N-1}$ is nonsingular, the linear system $\widehat{\bm{y}}\,\widehat{H}_{N-1}
=
\bm{e}_1^{\top}$ admits a unique nontrivial solution $\widehat{\bm{y}}$. We then define $\bm{y}
=
\begin{bmatrix}
\widehat{\bm{y}} & 0
\end{bmatrix}$.
By construction, this vector satisfies
\begin{equation} \label{yasl}
\bm{y}
(H_N-\tilde{z}I)
=
\begin{bmatrix}
\times & 1 & 0 & \cdots & 0
\end{bmatrix}.
\end{equation}
Combining the relations satisfied by the left eigenvector
$\bm{\dot{x}}_{le}$ and the auxiliary vector $\bm{y}$, we obtain
\[
\begin{bmatrix}
\bm{\dot{x}}_{le} \\ \bm{y}
\end{bmatrix}
(H_N-\tilde{z}I_{N})
=
\begin{bmatrix}
0 \;\; 0 \;\; 0  \;\; \cdots \;\; 0 \\
\times \;\; 1 \;\; 0 \;\; \cdots \;\; 0
\end{bmatrix},
\]
The relations satisfied by \(\bm{\dot{x}}_{le}\) and \(\bm{y}\) form the basis for constructing sequences of upper triangular eliminators that successively remove the first and second upper diagonals of the shifted matrix \(H_N-\tilde{z}I_N\). In the present section, we use a different class of elementary eliminators from those employed in the updating procedure. For \(i=1,\ldots,N-1\), matrices of the form
\begin{equation}\label{elim}
G_i =
\begin{bmatrix}
I_{i-1} &        &        &        \\
        & 1      &        &        \\
        & \xi_i  & 1      &        \\
        &        &        & I_{N-i-1}
\end{bmatrix},
\qquad
E_i =
\begin{bmatrix}
I_{i-1} &        &        &        \\
        & 1      & \eta_i &        \\
        &        & 1      &        \\
        &        &        & I_{N-i-1}
\end{bmatrix},
\end{equation}
are referred to as lower and upper triangular eliminators used in the remainder of the paper.

Now, we construct a sequence of upper triangular eliminators
$E_{N-2},E_{N-3},\ldots,E_1$ so that $\bm{y}$ is transformed
into $\bm{e}_1^{\top}
=
\begin{bmatrix}
1 & 0 & 0 & \cdots & 0
\end{bmatrix}$. More precisely,
\[
    \bm{\tilde{y}}
    =
    \bm{y}\mathcal{E}
    =
    \pm \bm{e}_1^{\top},
    \qquad
    \mathcal{E}=E_{N-2}E_{N-3}\cdots E_1.
\]
The last component of $\widehat{\bm{y}}$, denoted by $y_{N-1}$, is nonzero. Indeed, if $y_{N-1}=0$, then, since no breakdowns occur and the corresponding upper-diagonal entries of the recurrence matrix are nonzero, relation~\eqref{yasl} implies recursively that
$\bm{y}\bm{e}_j=0$, $j=1,2,\ldots,N-2$, contradicting the nontriviality of $\widehat{\bm{y}}$. Moreover, it follows from $\bm{y}(H_{N}-\tilde{z} I_{N})=\begin{bmatrix}
\times & 1 & 0 & \cdots & 0
\end{bmatrix}$ that $\begin{bmatrix}
y_{N-2} & y_{N-1}
\end{bmatrix}
\begin{bmatrix}
h_{N-2,N}\\
h_{N-1,N}
\end{bmatrix}
=0$.
The orthogonality of these two vectors implies that the upper triangular eliminator $E_{N-2}$ eliminating $y_{N-1}$ in the product
$\bm{y}E_{N-2}$ is related to the eliminator removing
$h_{N-2,N}$ in the product $E_{N-2}^{-1}(H_N-\tilde{z}I_{N})$. We then obtain the expression
\[
(\bm{y}E_{N-2})\bigl(E_{N-2}^{-1}(H_N-\tilde{z}I_{N})\bigr)
=
\left[
\begin{array}{ccccc}
\times &
\cdots &
\times &0&
0
\end{array}
\right]\left[
\begin{array}{cccccc}
\times & \times & \otimes &  &  &  \\
\times & \times & \times & \ddots &  &  \\
 & \ddots & \ddots & \ddots & \otimes &  \\
 &  & \times & \times & \times & 0 \\
 &  &  & \times & \times & \times \\
 &  &  &  & \times & \times
\end{array}
\right].
\]
Repeating the same argument inductively shows that the eliminators transforming the vector
$\bm{y}$ into $\bm{y}\mathcal{E}=\pm \bm{e}_1^{\top}$ are the same ones transforming the shifted matrix into a Hessenberg
matrix with the second upper diagonal removed, namely
\[\mathcal{E}^{-1}(H_N-\tilde{z}I_{N})=
\left[
\begin{array}{cccccc}
\times & \times &  &  &  &  \\
\times & \times & \times &  &  &  \\
 & \ddots & \ddots & \ddots &  &  \\
 &  & \times & \times & \times &  \\
 &  &  & \times & \times & \times \\
 &  &  &  & \times & \times
\end{array}
\right].
\]
To provide a clearer overview of the aforementioned procedure, we summarize the action of the $\mathcal{E}$ as
\[
\begin{bmatrix}
\bm{\dot{x}}_{le} \\[0.2em]
\bm{y}
\end{bmatrix}
\mathcal{E}\bigl(\mathcal{E}^{-1}(H_N-\tilde{z}I)\bigr)
=
\begin{bmatrix}
0 \;\; 0 \;\; 0 \;\; \cdots \;\; 0 \\
\times \;\; 1 \;\; 0 \;\; \cdots \;\; 0
\end{bmatrix}.
\]
We now construct another sequence of upper triangular eliminators $F_{N-1},F_{N-2},\ldots,F_1$,
which transforms the $\bm{\dot{x}}_{le} \mathcal{E}$ into $
\bm{e}_1^{\top} = \begin{bmatrix} 1 & 0 & \cdots & 0 \end{bmatrix}$. More precisely, these eliminators are chosen so that
\[
    \bm{\tilde{x}}_{le} = \bm{\dot{x}}_{le} \mathcal{E}\mathcal{F}=\pm \bm{e}_1^{\top},
    \qquad
    \mathcal{F}=F_{N-1}F_{N-2}\cdots F_1.
\]
Repeating the same argument as above shows that the eliminators
transforming the vector $\bm{\dot{x}}_{le}\mathcal{E}$ into $\bm{\dot{x}}_{le}\mathcal{E}\mathcal{F}
=
\pm \bm{e}_1^{\top}$ are the same eliminators that reduce the shifted matrix
$\mathcal{E}^{-1}(H_N-\tilde{z}I_{N})$ to the lower bidiagonal matrix
\[
\mathcal{F}^{-1}\mathcal{E}^{-1}(H_N-\tilde{z}I_{N})
=
\begin{bmatrix}
0 &        &        &        \\
\times & \times &        &        \\
       & \ddots & \ddots &        \\
       &        & \times & \times
\end{bmatrix}=\mathcal{L}_{1},
\]
with the zero in position $(1,1)$, the appearance of which will be explained later. Additional details on the elimination process can be found in the work of Mastronardi and Van Dooren~\cite[Theorem~2.1]{MR3867618}. 

The overall effect of \(\mathcal E\) and \(\mathcal F\) is summarized by
\[
\begin{bmatrix}
\bm{\dot{x}}_{le} \\[0.2em]
\bm{y}
\end{bmatrix}
\mathcal{E}\mathcal{F}\bigl(\mathcal{F}^{-1}\mathcal{E}^{-1}(H_N-\tilde{z}I)\bigr)
=
\begin{bmatrix}
0 \;\; 0 \;\; 0 \;\; \cdots \;\; 0 \\
\times \;\; 1 \;\; 0 \;\; \cdots \;\; 0
\end{bmatrix}.
\]
By setting $\mathcal{R}_{1}=\mathcal{E}\mathcal{F}$, an RL step for the Hessenberg matrix $H_{N}$ with shift $\tilde{z}$ is defined by $H_{N}-\tilde{z}I_{N}=\mathcal{R}_{1}\mathcal{L}_{1}$. It then follows that
\[
\widetilde{H}_{N}^{[1]}
=\mathcal{R}_{1}^{-1}H_{N}\mathcal{R}_{1}
=\mathcal{L}_{1}\mathcal{R}_{1}+\tilde{z}I_{N},
\]
which remains in upper Hessenberg form. Hence, this update can be interpreted as a similarity transformation of $H_{N}$ induced by $\mathcal{R}_{1}$. This transforms the pair $(H_{N},\bm{\dot{x}}_{le})$ to a similar one,
\[
    (\widetilde{H}_{N}^{[1]}, \bm{\tilde{x}}_{le}) = (\mathcal{R}_{1}^{-1} H_{N} \mathcal{R}_{1},\; \bm{\dot{x}}_{le}\mathcal{R}_{1}).
\]

To further clarify the structure of $\mathcal R_1$, $\mathcal L_1$, and the updated matrix $\widetilde{H}_{N}^{[1]}$, we now provide some additional insight into the RL construction.

The shifted Hessenberg matrix $H_{N}-\tilde{z} I_{N}$ is singular. Therefore, in the RL factorization $
H_{N}-\tilde{z} I_{N}=\mathcal{R}_{1}\,\mathcal{L}_{1}$,
the lower triangular factor $\mathcal{L}_{1}$ must also be singular.  Consequently, at least one diagonal entry of $\mathcal{L}_{1}$ should be zero. Since $H_{N}-\tilde{z} I_{N}$ is an unreduced Hessenberg matrix, its last $N-1$ rows are linearly independent. Moreover, since $\mathcal{R}_{1}$ is a product of nonsingular eliminators, it is itself nonsingular. Hence, its rows are linearly independent. It therefore follows that the last $N-1$ rows of $\mathcal{L}_{1}$ must also be linearly independent. Hence the only possible vanishing diagonal entry of $\mathcal{L}_{1}$ is the first one. This reveals that the first row of \(\mathcal{R}_{1}^{-1}\) in the RL factorization of \(H_{N} - \tilde{z} I_{N}\) is the left eigenvector corresponding to \(\tilde{z}\), and we also have 
\[
\bm{e}_1^{\top} (\widetilde{H}_{N}^{[1]}-\tilde{z}I_{N})=0.
\]
As a consequence, the updated matrix $\widetilde{H}_{N}^{[1]}$ attains the following structure
\[\widetilde{H}_{N}^{[1]}= \left[ \begin{array}{c|cccccc} \tilde{z}& & & & & & \\ \hline \otimes & \times & \times & \times & & & \\ & \times & \times & \times & \ddots & & \\ & & \ddots & \ddots & \ddots & \times & \\ & & & \times & \times & \times & \times \\ & & & & \times & \times & \times \\ & & & & & \times & \times \end{array} \right]. \]
Although the first perfect RL step moves the eigenvalue \(\tilde z\) to the leading diagonal position $(1,1)$, the eigenvalue is not yet completely isolated, since the entry in position \((2,1)\) of the updated matrix \(\widetilde{H}_{N}^{[1]}\) generally remains nonzero. To completely isolate the eigenvalue \(\tilde z\), we therefore perform a second perfect RL step. In contrast to the previous construction, which was based on the left eigenvector, this step is built from the corresponding right eigenvector $\bm{x}_{ri}$. As before, the accuracy of the eigenvector can be improved by performing a few steps of iterative refinement, yielding a refined eigenvector \(\bm{\dot{x}}_{ri}\), following the refinement procedure described at the beginning of this section.
In the next step, which corresponds to the perfect RL step, a sequence of lower triangular eliminators $G_{N-1} \cdots G_2 G_1 $ is performed to transform $\bm{\dot{x}}_{ri}$ into $\bm{e}_1 = \begin{bmatrix} 1 & 0 & \cdots & 0 \end{bmatrix}^{T}$. Define the lower triangular matrix $\mathcal{L}_{2}= \prod_{i=N-1}^{1} G_i$. The eliminators transforming the vector $\bm{\dot{x}}_{ri}$ into $\mathcal{L}_{2}\bm{\dot{x}}_{ri} =\bm{e}_1$ are the same eliminators that reduce the shifted matrix
$(\widetilde{H}_{N}^{[1]}-\tilde{z}I_{N})$ to the upper triangular matrix
\[
(\widetilde{H}_{N}^{[1]}-\tilde{z}I_{N})\mathcal{L}_{2}^{-1}
=
\begin{bmatrix}
0 & \times & \times &        &        \\
  & \times & \times & \ddots &        \\
  &         & \ddots & \ddots & \times \\
  &         &        & \times & \times \\
  &         &        &        & \times
\end{bmatrix}=\mathcal{R}_{2},
\]
with the zero in position $(1,1)$. Moreover, since $\mathcal L_2 \bm{\dot{x}}_{ri}=\bm e_1$, the first column of \(\mathcal L_2^{-1}\) coincides with the right eigenvector associated with \(\tilde z\). As a result, a similarity transformation with $\mathcal{L}_{2}$ yields
\begin{equation}\label{properH}
\widetilde{H}_{N}^{[2]} =\mathcal{L}_{2} \widetilde{H}_{N}^{[1]} \mathcal{L}_{2}^{-1} \;=\;
\left[ \begin{array}{c|cccccc} \tilde{z}& & & & & & \\ \hline & \times & \times & \times & & & \\ & \times & \times & \times & \ddots & & \\ & & \ddots & \ddots & \ddots & \times & \\ & & & \times & \times & \times & \times \\ & & & & \times & \times & \times \\ & & & & & \times & \times \end{array} \right]=\mathcal{L}_{2}\mathcal{R}_{2}+\tilde{z}I_{N}.
\end{equation}
This completes the isolation of the eigenvalue \(\tilde z\), which is now fully separated in the leading \(1\times1\) block and therefore ready for deflation.
\subsection{Effect of the RL steps on the biorthonormal bases}\label{subdown}
Having isolated the eigenvalue \(\tilde z\), we now examine the effect of the perfect RL steps on the biorthonormal pair \((V_N,W_N)\). The structured similarity transformations constructed above naturally induce transformations of the associated basis matrices. From these transformed bases, we extract the reduced biorthonormal pair \((\widetilde V_{N-1},\widetilde W_{N-1})\) satisfying \ref{cond:222}. 

Combining the previous constructions with the eigenvalue decomposition of the diagonalizable matrix $H_{N}$, we arrive at
\begin{eqnarray}\label{dd33}
	\widetilde{H}_{N}^{[2]}=\mathcal{L}_{2} \mathcal{R}_{1}^{-1} \, H_{N} \, \mathcal{R}_{1} \mathcal{L}_{2}^{-1}=\mathcal{L}_{2} \mathcal{R}_{1}^{-1}\, W_{N}^{\top} \, Z \, V_{N}\, \mathcal{R}_{1} \mathcal{L}_{2}^{-1}.
\end{eqnarray}
We have that the first row of $\mathcal{R}_{1}^{-1}$, denoted by \(\bm r_1\), is equal to the \(j\)-th row \(\bm v_j\) of \(V_N\), for some $j$, which is a left eigenvector of $H_{N}$ associated with $\tilde{z}$. Because of the biorthogonality of the left and right eigenvectors, we obtain $\bm{r}_{1}\bm{w}_{i}^{\top} = 0$ for $i \neq j$ and $\bm{r}_{1}\bm{w}_{j}^{\top} = 1$. Partitioning \(\mathcal R_1^{-1}\) as
\[
\mathcal R_1^{-1}
=
\begin{bmatrix}
\bm r_1\\
R_{N-1}
\end{bmatrix},
\qquad
R_{N-1}\in\mathbb R^{(N-1)\times N},
\] we obtain
\begin{equation*}
\mathcal{R}_{1}^{-1}\, W_{N}^{\top}
= 
\begin{bmatrix}
\bm{r}_{1}\\ R_{N-1}
\end{bmatrix}
\begin{bmatrix}
\bm{w}_{i}^{\top}
\end{bmatrix}_{i=1}^{N}
=
\begin{bmatrix}
\text{\textemdash}& \bm{e}_{j}^{\top}& \text{\textemdash}\\
 & \bm{\times} & \\
\end{bmatrix},
\end{equation*}
where $\bm{w}_i$ denotes the $i$-th row of $W_{N}$, and $\bm{\times} \in \mathbb{R}^{(N-1) \times N}$ collects arbitrary, nonessential entries. Based on $\mathcal{R}_{1}^{-1}\mathcal{R}_{1}=I_{N}$, together with $\bm{r}_{1}=\bm{v}_{j}$, we obtain $\bm{v}_{j} \mathcal{R}_{1} \bm{e}_{1} = 1$ and $\bm{v}_{j}\mathcal{R}_{1} \bm{e}_{i}=0$ for $i \ne 1$, which yields
\[
V_{N}\, \mathcal{R}_{1}
=
\begin{bmatrix}
\bm{w}_{i}^{\top}
\end{bmatrix}_{i=1}^{N}
\mathcal{R}_{1}
=
\begin{bmatrix}
 & \bm{\times} & \\
\text{\textemdash} & \bm{e}_{1}^{\top} & \text{\textemdash} \\
 & \bm{\times} & \\
\end{bmatrix}.
\]

Writing the inverse matrix \(\mathcal L_2^{-1}\) in block form as
\[
\mathcal L_2^{-1}
=
\begin{bmatrix}
\bm \ell_1 & L_{N-1}
\end{bmatrix},
\qquad
L_{N-1}\in\mathbb R^{N\times (N-1)},
\] the first column of \(\mathcal L_2^{-1}\), denoted by \(\bm \ell_1\), coincides with the \(j\)-th right eigenvector \(\bm w_j^{\top}\) associated with \(\tilde z\). By biorthogonality, $\bm{e}_{j}^{\top} (V_{N} \mathcal{R}_{1})\bm{\ell}_{1} = 1$ and $\bm{e}_{j}^{\top} (V_{N} \mathcal{R}_{1}) L_{N-1}=\bm{0}$. Consequently,
\[
V_{N}\, \mathcal{R}_{1}\mathcal L_2^{-1}
=
\begin{bmatrix}
 & \bm{\times} & \\
\text{\textemdash} & \bm{e}_{1}^{\top} & \text{\textemdash} \\
 & \bm{\times} & \\
\end{bmatrix}
\begin{bmatrix}
\bm{\ell}_{1} & L_{N-1}
\end{bmatrix}
=
\begin{bmatrix}
\mid & \bm{\times}  \\
\bm{e}_{j}& \bm{0}  \\
 \mid& \bm{\times}  \\
\end{bmatrix},
\]
where $\bm{0}$ denotes the zero row vector of length $N-1$. Using the relation $\mathcal{L}_{2}\mathcal{L}_{2}^{-1}=I_{N}$, together with $\bm{\ell}_{1}=\bm{v}_{j}$, we obtain $\bm{e}_{i}^{\top} \mathcal{L}_{2}(\mathcal{R}_{1}^{-1}\, W_{N}^{\top})\bm{e}_{j}= 0$ for $i \neq j$ and $\bm{e}_{j}^{\top} \mathcal{L}_{2}(\mathcal{R}_{1}^{-1}\, W_{N}^{\top})\bm{e}_{j}= 1$, which yields
\[
\mathcal{L}_{2}(\mathcal{R}_{1}^{-1}\, W_{N}^{\top})=\mathcal{L}_{2}\begin{bmatrix}
\text{\textemdash}& \bm{e}_{j}^{\top}& \text{\textemdash}\\
 & \bm{\times} & \\
\end{bmatrix}=\begin{bmatrix}
\text{\textemdash}& \bm{e}_{j}^{\top}& \text{\textemdash}\\
\bm{\times} & \bm{0} & \bm{\times}\\
\end{bmatrix},
\]
where \(\bm 0\) denotes the zero column vector of length \(N-1\). 

In view of \eqref{dd33}, we define the transformed biorthonormal pair
\begin{equation*}
V_{N}^{[1]}=V_{N}\mathcal{R}_{1}\mathcal{L}_{2}^{-1},
\qquad
W_{N}^{[1]}=W_{N}\mathcal{R}_{1}^{-\top}\mathcal{L}_{2}^{\top}.
\end{equation*}

The structural properties established above will now be used to construct the reduced biorthonormal pair satisfying condition~\ref{cond:222}.
\subsection{Verification of the basis conditions}
It remains to verify condition~\ref{cond:111}. To this end, we explicitly investigate the effect of the transformations on the first column of \(V_N\) and the first and second columns of \(W_N\). As the pair of biorthonormal matrices 
\( V_{N}, W_{N} \) satisfies \ref{cond:1}, we have that 
\[
\mathcal{X} \mathcal{R}
	= W_{N} \begin{bmatrix}
		\mid & \mid \\
		\bm{e}_1 & \bm{e}_2 \\
		\mid & \mid
	\end{bmatrix},\ \quad \text{and} \quad
	\begin{bmatrix}
		\mid \\
		\bm{v}_1 \\
		\mid 
	\end{bmatrix}
	= V_{N} \begin{bmatrix}
		\mid \\
		\bm{e}_1 \\
		\mid 
	\end{bmatrix},
\]
or equivalently
\begin{equation}\label{first}
W_{N}^{-1}\mathcal{X} 
	=  \begin{bmatrix}
		\mid & \mid \\
		\bm{e}_1 & \bm{e}_2 \\
		\mid & \mid
	\end{bmatrix}\mathcal{R}^{-1},\ \quad \text{and} \quad
	V_{N}^{-1}\begin{bmatrix}
		\mid \\
		\bm{v}_1 \\
		\mid 
	\end{bmatrix}
	=  \begin{bmatrix}
		\mid \\
		\bm{e}_1 \\
		\mid 
	\end{bmatrix}.
\end{equation}
Applying the transformations defining \(V_N^{[1]}\) and \(W_N^{[1]}\) to the relations in~\eqref{first}, we obtain
\begin{equation*}
\mathcal{L}_{2}^{-\top}\mathcal{R}_{1}^{\top}W_{N}^{-1}\mathcal{X} 
	=  \mathcal{L}_{2}^{-\top}\mathcal{R}_{1}^{\top}\begin{bmatrix}
		\mid & \mid \\
		\bm{e}_1 & \bm{e}_2 \\
		\mid & \mid
	\end{bmatrix}\mathcal{R}^{-1},\ \quad  \quad
	\mathcal{L}_{2}\mathcal{R}_{1}^{-1}V_{N}^{-1}\begin{bmatrix}
		\mid \\
		\bm{v}_1 \\
		\mid 
	\end{bmatrix}
	=  \mathcal{L}_{2}\mathcal{R}_{1}^{-1}\begin{bmatrix}
		\mid \\
		\bm{e}_1 \\
		\mid 
	\end{bmatrix}.
\end{equation*}
 The structure of \(\mathcal R_1\) and \(\mathcal L_2\) now allows us to characterize
 \begin{equation*}
\left(W_N^{[1]}\right)^{-1}\mathcal{X} 
	= \underbrace{\left[
\begin{array}{cc}
\times & \times \\ \hline
\times & \times \\
\otimes & \times \\
0 & \otimes \\
\bm 0 & \bm 0
\end{array}
\right]}_{=\bm {M}}\mathcal{R}^{-1},\ \quad  \quad
	\left(V_N^{[1]}\right)^{-1}\begin{bmatrix}
		\mid \\
		\bm{v}_1 \\
		\mid 
	\end{bmatrix}
	=\underbrace{\left[
\begin{array}{c}
\times \\ \hline
\times \\
\bm 0
\end{array}
\right]}_{=\bm{m}},
\end{equation*}
where \(\bm 0\) denotes a zero vector of appropriate dimension. The first row of \(\bm M\) and the first entry of \(\bm m\) are separated by a horizontal line, since these components correspond to the deflated part of the problem and will eventually be removed. Moreover, the matrix \(\bm M\) contains two distinguished nonzero entries, indicated by \(\otimes\). These entries must be eliminated in order to recover, after deflation, the desired structure $\begin{bmatrix}
\mid & \mid \\
\bm e_1 & \bm e_2 \\
\mid & \mid
\end{bmatrix}
\mathcal R_{\mathrm{down}}^{-1}$, where \(\bm e_1,\bm e_2\in\mathbb R^{N-1}\). To eliminate these two entries, we apply the lower triangular matrix \(\mathcal L_3=G_3G_2\), defined as the product of the lower triangular eliminators \(G_2\) and \(G_3\) from \eqref{elim}, resulting in
 \begin{equation}\label{Weights}
\mathcal{L}_{3}\left(W_N^{[1]}\right)^{-1}\mathcal{X} 
	= \left[
\begin{array}{cc}
\times & \times \\ \hline
\times & \times \\
0 & \times \\
\bm 0 & \bm 0
\end{array}
\right]\mathcal{R}^{-1},\ \quad  \quad
	\mathcal{L}_{3}^{-\top}\left(V_N^{[1]}\right)^{-1}\begin{bmatrix}
		\mid \\
		\bm{v}_1 \\
		\mid 
	\end{bmatrix}
	= \left[
\begin{array}{c}
\times \\ \hline
\times \\
\bm 0
\end{array}
\right].
\end{equation}
This additional transformation induces a new biorthonormal pair and the corresponding transformed Hessenberg matrix,
\[
V_N^{[2]} = V_N^{[1]}\mathcal L_3^{\top},
\qquad
W_N^{[2]} = W_N^{[1]}\mathcal L_3^{-1},
\qquad
H_N^{[3]}
=
\mathcal{L}_3^{-\top} H_N^{[2]} \mathcal L_3^{\top}.
\]
However, the matrix \(H_N^{[3]}\) now takes the form
\[
\left[
\begin{array}{c|cccccccc}
\tilde z &        &        &        &        &        &        &        \\ \hline
         & \times & \times & \times & \otimes &        &        &        \\
         & \times & \times & \times & \times  & \otimes &        &        \\
         &        & \times & \times & \times  & \times  &        &        \\
         &        &        & \ddots & \ddots  & \ddots  & \ddots &        \\
         &        &        &        & \times  & \times  & \times & \times \\
         &        &        &        &         & \times  & \times & \times \\
         &        &        &        &         &         & \times & \times
\end{array}
\right],
\]
where the entries marked by \(\otimes\), located at positions \((2,5)\) and \((3,6)\), must be chased further in order to recover the desired structure~\eqref{properH}. To restore the \(4\)-banded Hessenberg structure in the trailing block, a sequence of upper triangular eliminators is constructed to eliminate the bulges. Eliminating these entries generates new unwanted elements on the third superdiagonal of \(H_N^{[3]}\), which are subsequently chased downward along the matrix. In general, this requires \(N-4\) upper triangular eliminators to remove all entries lying above the second superdiagonal, thereby recovering the desired structure without altering the first column of \(V_N^{[2]}\) or the first two columns of \(W_N^{[2]}\). Since this bulge-chasing procedure is analogous to the one used in the updating process and in Faghih et al.~\cite[Section~5]{Faghih2025KrylovCore}, we omit the technical details. Denoting the corresponding eliminators by \(P_s\), \(s=1,2,\ldots,N-4\), we define the upper triangular matrix $
\mathcal R_2 = P_1P_2\cdots P_{N-4}$. Under this transformation, the biorthonormal bases and the associated Hessenberg matrix are updated as
\[
V_N^{[3]} = V_N^{[2]}\mathcal R_2,
\qquad
W_N^{[3]} = W_N^{[2]}\mathcal R_2^{-\top},
\qquad
H_N^{[4]}
=
\mathcal R_2^{-1} H_N^{[3]} \mathcal R_2=\begin{bmatrix}
\tilde z & \\
& \widetilde H_{N-1}
\end{bmatrix} .
\]
The biorthonormal matrices \(V_N^{[3]}\) and \(W_N^{[3]}\) possess the same structural properties as \(V_N^{[1]}\) and \(W_N^{[1]}\) derived in Subsection~\ref{subdown}. Introducing the permutation matrix \(P\), which interchanges rows \(1\) and \(j\), therefore yields
\[
  (PW_N^{[3]})^{\top}\,
  (P Z P^{\top})\,
  (PV_N^{[3]} )
  =
\begin{bmatrix}
1 &\\
& \widetilde{W}_{N-1}^{\top}
\end{bmatrix}
\begin{bmatrix}
\tilde{z} & \\
& \widetilde{Z}
\end{bmatrix}
\begin{bmatrix}
1 &\\
& \widetilde{V}_{N-1}
\end{bmatrix}
=
\begin{bmatrix}
\tilde{z} & \\
& \widetilde{H}_{N-1}
\end{bmatrix}.
\]
The matrices \(\widetilde{W}_{N-1}\), \(\widetilde{Z}\), \(\widetilde{V}_{N-1}\), and \(\widetilde{H}_{N-1}\) are therefore the downdated factors obtained after deflating the isolated eigenvalue \(\tilde z\). Consequently, this provides a solution to Problem~\ref{Downdate:IEP}, satisfying conditions~\ref{cond:222} and~\ref{cond:333}. The total structured similarity transformation in \eqref{dd} is therefore given by
\[
\mathcal S
=
\mathcal R_1\mathcal L_2^{-1}\mathcal L_3^{\top}\mathcal R_2 .
\]
Applying the permutation matrix \(P\) and the transformation \(\mathcal R_2\) to \eqref{Weights} yields
 \begin{equation*}
\left(PW_N^{[3]}\right)^{-1}(P\mathcal{X}) 
	= \left[
\begin{array}{cc}
\times & \times \\ \hline
\times & \times \\
0 & \times \\
\bm 0 & \bm 0
\end{array}
\right]\mathcal{R}^{-1},\ \quad  \quad
	\left(PV_N^{[3]}\right)^{-1}P\begin{bmatrix}
		\mid \\
		\bm{v}_1 \\
		\mid 
	\end{bmatrix}
	= \left[
\begin{array}{c}
\times \\ \hline
\times \\
\bm 0
\end{array}
\right].
\end{equation*}
The permutation matrix \(P\) moves the entries \(\alpha_{1,j}\) and \(\alpha_{2,j}\) to the leading positions of \(\mathcal X\), thereby preparing them for deflation. After removing the first row and column corresponding to the isolated eigenvalue, the right-hand sides reduce to
$\begin{bmatrix}
\mid & \mid \\
\bm e_1 & \bm e_2 \\
\mid & \mid
\end{bmatrix}
\mathcal R_{\mathrm{down}}^{-1}$, and $\bm e_1$, respectively, 
with \(\bm e_1,\bm e_2\in\mathbb R^{N-1}\). Hence, condition~\ref{cond:111} is satisfied. This follows from the uniqueness of the solution established in \cite[Proposition~4.2]{Faghih2025KrylovCore}.

Finally, the scaling strategy described in Section~\ref{Scaling} is applied to 
$\widetilde{W}_{N-1}$, $\widetilde{V}_{N-1}$, and $\widetilde{H}_{N-1}$, 
producing the scaled solution 
$\widetilde{W}_{N-1}^{S}$, $\widetilde{V}_{N-1}^{S}$, and $\widetilde{H}_{N-1}^{S}$ 
for the downdated IEP~\ref{Downdate:IEP}.
\section{Numerical experiments}\label{SSec5}
In this section, we present some numerical experiments to assess the performance of the updating and downdating algorithms. All computations were performed in MATLAB, and the corresponding code is publicly available.\footnote{\url{https://wms.cs.kuleuven.be/groups/NUMA/software/momentum-software}}

Assume that the Hessenberg matrix $\widetilde{H}$, together with a pair of
biorthonormal matrices $\widetilde{V}, \widetilde{W}$, are obtained as the
solution of Problem~\ref{Update:IEP} or Problem~\ref{Downdate:IEP} via the
updating and downdating procedures. The dimensions of the matrices $\widetilde{H}$, $\widetilde{V}$, and $\widetilde{W}$ depend on whether an updating or a downdating procedure is applied. For notational simplicity, matrix dimensions are omitted in the following definitions and will be specified explicitly in the numerical examples. The following metrics are used for the errors of the computed solutions.
\begin{itemize}
\item The accuracy of the recurrence relation, consisting of the recurrence
matrix $\widetilde{H}$ and biorthonormal basis $(\widetilde{V},\widetilde{W})$, is measured by
\begin{equation*}
E_r=\dfrac{\big\Vert \widetilde{W}^{\top} \widetilde{Z} \widetilde{V}-\widetilde{H}\big\Vert_{2}}{\max \big(\big\Vert \widetilde{W}^{\top} \widetilde{Z} \widetilde{V}\big\Vert_{2},\big\Vert \widetilde{H}\big\Vert_{2}\big)},
\end{equation*}
where $\widetilde{Z}$ denotes the appropriately redefined matrix of nodes, i.e., $Z$ where the downdated nodes are omitted or the updated nodes are added.\\ 

\item The bi-orthonormality of the formed basis $(\widetilde{V}, \widetilde{W})$ is measured by
	\begin{equation*}
	E_{bo} = \big\Vert \widetilde{W}^{\top}\widetilde{V} - I\big\Vert_{2}.
	\end{equation*}
\end{itemize}

\subsection{Updating}
As the first and second examples, we consider two families of discrete multiple orthogonal polynomials, the Kravchuk and Hahn MOPs \cite{MR1985676}.

\begin{example}[Multiple Kravchuk polynomials \cite{MR1985676}]
The Kravchuk polynomials are associated with discrete measures defined by binomial
distributions supported on the integers. More precisely, we consider
\[
\omega_j = \sum_{i=0}^{N-1} \binom{N-1}{i}\, p_j^{\,i}(1-p_j)^{N-1-i}\,\delta_i, \qquad j=1,2,
\]
where \(0 < p_j < 1\) are distinct. Throughout
our numerical experiments, we fix \(p_1 = 0.4\) and \(p_2 = 0.5\).

Explicit formula for the type~II multiple Kravchuk polynomials
\(K_{n_1,n_2}^{p_1,p_2,N}(x)\) for a multi-index \((n_1,n_2)\), is given by \cite[Section~4.4]{MR1985676}
\begin{multline*}
K_{n_1,n_2}^{p_1,p_2,N}(x)
= p_1^{n_1} p_2^{n_2} (-N+1)_{n_1+n_2}
\sum_{j=0}^{n_1+n_2} \sum_{k=0}^{j}
\frac{(-n_1)_k}{k!}
\left(\frac{1}{p_1}\right)^k
\frac{(-n_2)_{j-k}}{(j-k)!}
\left(\frac{1}{p_2}\right)^{j-k}
\frac{(-x)_j}{(-N+1)_j}.
\end{multline*}
Here, \((c)_j\) denotes the Pochhammer function, defined by
\((c)_j = \prod_{i=0}^{j-1} (c+i)\) for \(j>0\), with \((c)_0 = 1\).
\end{example}
\begin{example}[Multiple Hahn polynomials \cite{MR1985676}]
Hahn polynomials arise when the discrete measures form a
hypergeometric distribution on the integers. More precisely, we consider
\[
\omega_j = \sum_{i=0}^{N-1}
\frac{(\beta_j+1)_i}{i!}
\frac{(\gamma+1)_{N-i-1}}{(N-1-i)!}
\,\delta_i, \qquad j=1,2,
\]
where the parameters satisfy \(\beta_j > -1\) and \(\gamma > -1\), and the values
\(\beta_1\) and \(\beta_2\) are distinct. In all numerical experiments, we fix
\(\beta_1 = 1\), \(\beta_2 = 1.5\), and \(\gamma = 1\).
For this AT-system, explicit formulae for the corresponding multiple Hahn
polynomials are available; see~\cite[Section~4.4]{MR1985676}.
\end{example}

In the case of Kravchuk and Hahn polynomials the weight vectors are not nested, which prevents incremental updating. Therefore, when increasing $N$, the computation cannot start from the solution of size $N-1$. Instead, for each $N$ the algorithm is restarted from the basic solution $H_{2}$ and $V_{2}, W_{2}$ to problem \ref{Prob:IEP} of size $2 \times 2$\footnote{The basic \(2\times 2\) solution is computed using the core transformation method developed in~\cite{Faghih2025KrylovCore}.}, 
and the full procedure is applied to compute the recurrence matrix 
and the bi-orthonormal pairs for the problem of size $N \times N$. We note that, in both cases, $P_{0}(z_{i})=1$. 

In Figure \ref{fig:Hahn_Krav}, we show the relative error $E_r$ and the loss of biorthogonality $E_{bo}$, over different numbers of nodes $N$, for the Kravchuk and Hahn MOPs.

This problem was also studied by Faghih et al.~\cite{Faghih2025KrylovCore} using core transformation and bi-orthogonal Lanczos methods.
In that work, the monicness condition is imposed on the type II MOPs, requiring $P_0(x)=1$ and, consequently, all subdiagonal entries of the recurrence matrix to be equal to one; see~\cite[Section~5]{Faghih2025KrylovCore}. In the present work, we do not impose the monicness condition, thereby allowing the subdiagonal entries of the recurrence matrix to differ from one, and instead employ a scaling strategy for the computation of the recurrence matrix and the associated bi-orthonormal matrix pairs.

In the present work, we solve the same problem using the proposed updating procedure in two settings: first, without the scaling strategy and with the subdiagonal entries constrained to be one, corresponding to the monic setting considered in~\cite{Faghih2025KrylovCore} and referred to as \emph{monic}; and second, with the scaling strategy introduced in Section~\ref{Scaling}, referred to as \emph{scaled}. The numerical comparison between these two settings is presented in Figure~\ref{fig:Hahn_Krav}, where the corresponding results are explicitly labeled.
\begin{figure}[h]
   \begin{subfigure}[b]{0.45\textwidth}
        \centering
		\setlength\figureheight{8cm}
		\setlength\figurewidth{13cm}
        \input{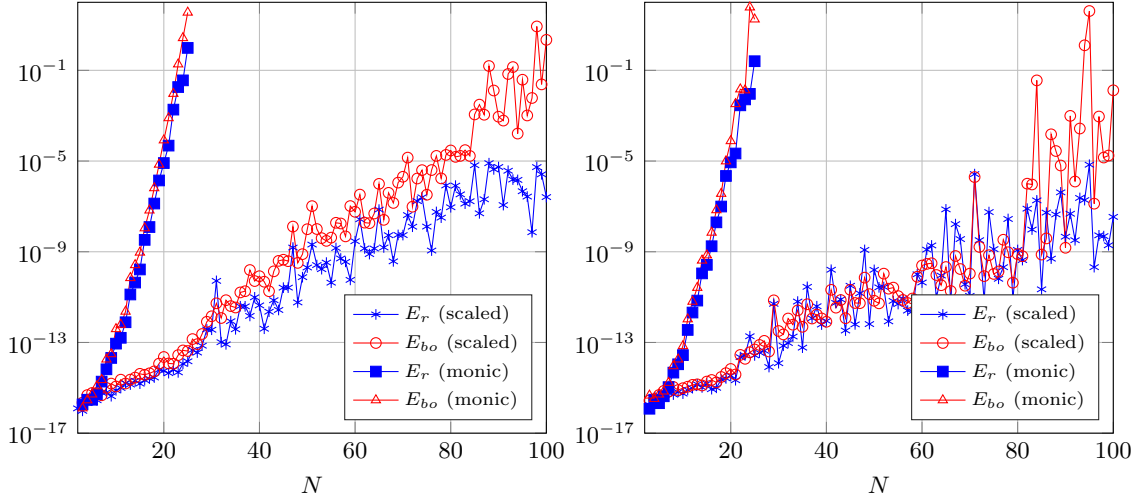}
    \end{subfigure}
    \hspace{0.05\textwidth}
    \caption{Error metrics for increasing the size of the sequence of Kravchuk (left) and Hahn (right) MOPs.}\label{fig:Hahn_Krav}
    \end{figure}

From these results, several observations can be made. The generation of Kravchuk and Hahn MOPs is known to be highly ill-conditioned, as analyzed in~\cite[Section~6]{Faghih2025KrylovCore}. This behavior is clearly reflected in Figure~\ref{fig:Hahn_Krav}, where essentially all decimal digits of accuracy are lost for values of $N$ between $20$ and $25$. 
When the monic setting of~\cite{Faghih2025KrylovCore} is employed, without applying the scaling strategy proposed in this work, both the recurrence error $E_r$ (monic) and the loss of biorthogonality $E_{bo}$ (monic) grow rapidly with increasing $N$, exhibiting an essentially exponential deterioration.

In contrast, the new scaling strategy introduced in Section~\ref{Scaling} leads to a significant improvement in numerical stability. 
For both Kravchuk and Hahn MOPs, the errors remain several orders of magnitude smaller and show only mild growth as $N$ increases. 
Even for $N = 100$, approximately five digits of accuracy are retained in the recurrence relation, and the bi-orthonormality error is relatively below $10^{-1}$. This demonstrates that the proposed scaling significantly mitigates the ill-conditioning of the problem and enables the stable computation of large sequences of MOPs. However, we note that the observed difference between the results arises from both the proposed scaling strategy and the nature of the updating procedure. In the updating approach, the nodes are added sequentially, starting from a smaller-size problem and incrementally building up to size $N$, while reusing the previously computed solution at each step. This, in fact, provides better control over the matrices computed from size $2 \times 2$ up to $N \times N$. 
In contrast, in the bi-orthogonal Lanczos and core transformation methods \cite{Faghih2025KrylovCore}, the full problem of size $N \times N$ is solved all at once.

To investigate the loss of biorthogonality observed in the computed pair of bases, we consider the following numerical experiment. First, we compute the approximate recurrence matrix $\widetilde{H}_{N}$. Using the recurrence coefficients contained in $\widetilde{H}_{N}$ together with the exact initial vectors $\bm{v}_{1}$, $\bm{w}_{1}$, and $\bm{w}_{2}$ obtained through Proposition~\ref{prop:initialvectors}, we regenerate the matrices $V_{N}$ and $W_{N}$ using the recurrence relation, for which ideally $W_{N}^{\top}\widetilde{Z}V_{N}=\widetilde{H}_{N}$.

In order to determine whether the loss of biorthogonality originates from inaccuracies in the computed recurrence coefficients or from the recurrence process itself, we repeat the same procedure using the exact recurrence matrix $H_{N}$ for both Kravchuk and Hahn polynomials \cite{MR1985676}. Starting again from the exact initial conditions, we generate the corresponding pair of biorthogonal matrices $V_{N}$ and $W_{N}$ using the recurrence relation.

For both cases, we measure the biorthogonality of the computed bases, and the results are shown in Figure~\ref{fig:Hahn_Krav_BO}.

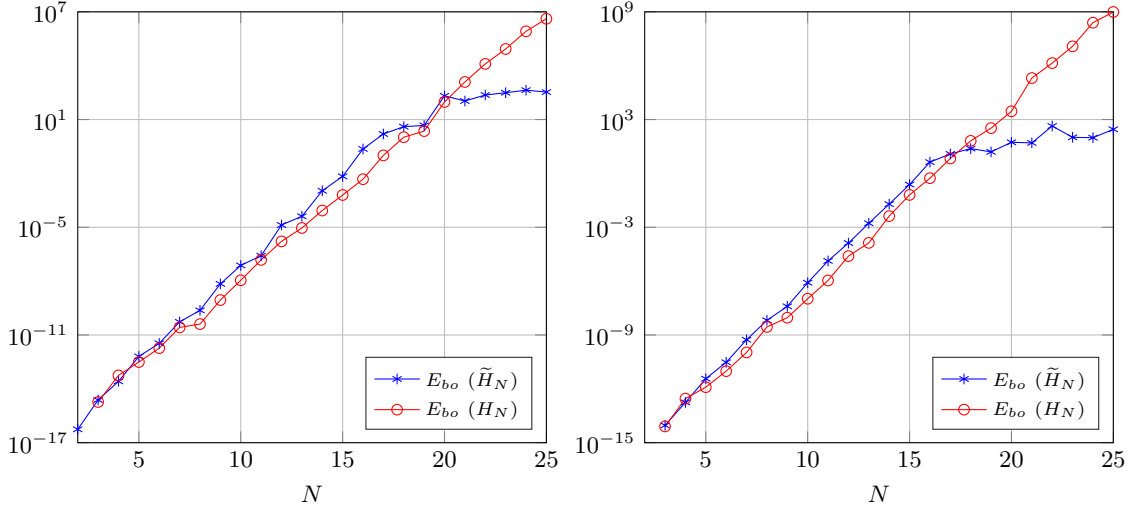
\begin{figure}[h]
   \begin{subfigure}[b]{0.45\textwidth}
        \centering
		\setlength\figureheight{8cm}
		\setlength\figurewidth{13cm}
%
\begin{tikzpicture}
\begin{axis}[%
width=6.2cm,
height=5.7cm,
at={(0cm,0cm)},
scale only axis,
xmin=2,
xmax=25,
xlabel style={font=\color{white!15!black}},
xlabel={$N$},
ymode=log,
ymin=1e-17,
ymax=10000000,
yminorticks=true,
axis background/.style={fill=white},
legend style={at={(0.97,0.03)}, anchor=south east, legend cell align=left, align=left, draw=white!15!black},
grid=major,
ticklabel style={font=\small},
label style={font=\small},
legend style={font=\scriptsize}
]
\addplot [color=blue, mark=asterisk, mark options={solid, fill=blue, blue}]
  table[row sep=crcr]{%
2	5.55111512312578e-17\\
3	2.39151168851531e-15\\
4	2.61388719397042e-14\\
5	6.37379244929381e-13\\
6	3.42510562899387e-12\\
7	5.3515222154564e-11\\
8	2.34961266470656e-10\\
9	7.02490884518894e-09\\
10	7.46408952098831e-08\\
11	2.60119319185451e-07\\
12	1.33638805956327e-05\\
13	4.08223131612404e-05\\
14	0.00107744536122074\\
15	0.00683451718118819\\
16	0.231258544701383\\
17	1.57254135591034\\
18	4.0156336985677\\
19	4.70587850196754\\
20	209.593345655283\\
21	108.488556203515\\
22	232.698521837511\\
23	312.078019369324\\
24	432.745041635493\\
25	338.012015496364\\
};
\addlegendentry{$E_{bo}$ ($\widetilde
{H}_{N}$)}

\addplot [color=red, mark=o, mark options={solid, red}]
  table[row sep=crcr]{%
2	0\\
3	1.85240653041116e-15\\
4	5.60016599667868e-14\\
5	3.12110036770384e-13\\
6	1.85031651008355e-12\\
7	2.64191841859968e-11\\
8	4.03732291644002e-11\\
9	8.8248995900118e-10\\
10	1.11434270627852e-08\\
11	1.52637383535497e-07\\
12	1.63347655355117e-06\\
13	9.12993127266349e-06\\
14	8.54440613379883e-05\\
15	0.000625935568670824\\
16	0.00474578244687727\\
17	0.100863085522701\\
18	1.04324901143283\\
19	2.27567236816284\\
20	95.3780372196583\\
21	1239.02082171435\\
22	12608.5370007029\\
23	85082.2734272641\\
24	824402.947198811\\
25	4119740.14889842\\
};
\addlegendentry{$E_{bo}$ ($H_{N}$)}

\end{axis}

\begin{axis}[%
width=6.2cm,
height=5.7cm,
at={(7.5cm,0cm)},
scale only axis,
xmin=2,
xmax=25,
xlabel style={font=\color{white!15!black}},
xlabel={$N$},
ymode=log,
ymin=1e-15,
ymax=1000000000,
yminorticks=true,
axis background/.style={fill=white},
legend style={at={(0.97,0.03)}, anchor=south east, legend cell align=left, align=left, draw=white!15!black},
grid=major,
ticklabel style={font=\small},
label style={font=\small},
legend style={font=\scriptsize}
]
\addplot [color=blue, mark=asterisk, mark options={solid, fill=blue, blue}]
  table[row sep=crcr]{%
2	0\\
3	9.21623154278725e-15\\
4	1.74466885884829e-13\\
5	3.67618484189278e-12\\
6	3.00873923379506e-11\\
7	5.40522767187642e-10\\
8	6.54357568330277e-09\\
9	3.94163613207289e-08\\
10	8.12271299787527e-07\\
11	1.33522812939227e-05\\
12	0.000131985033980042\\
13	0.00167346146655114\\
14	0.0194511203449801\\
15	0.239146313971288\\
16	4.25418428319422\\
17	12.14973723091\\
18	24.1295145910739\\
19	15.7872761743067\\
20	54.1597779880772\\
21	50.0425587823182\\
22	437.502392054832\\
23	102.270321931712\\
24	97.4601744871167\\
25	285.466337716606\\
};
\addlegendentry{$E_{bo}$ ($\widetilde{H}_{N}$)}

\addplot [color=red, mark=o, mark options={solid, red}]
  table[row sep=crcr]{%
2	0\\
3	8.05889475315074e-15\\
4	2.8602581451513e-13\\
5	1.23675265968682e-12\\
6	9.66628775738493e-12\\
7	1.06602698424214e-10\\
8	2.80730807348697e-09\\
9	9.17108962498768e-09\\
10	1.04548615355092e-07\\
11	1.08658339946208e-06\\
12	2.47479690427857e-05\\
13	0.000135228124925643\\
14	0.00416791510140659\\
15	0.0634827397280964\\
16	0.542752284973321\\
17	6.7658844470053\\
18	65.006295508226\\
19	333.422981314982\\
20	2886.72562009191\\
21	204919.664130825\\
22	1401492.46119856\\
23	11983428.6152984\\
24	248552114.905844\\
25	983935442.257661\\
};
\addlegendentry{$E_{bo}$ ($H_{N}$)}

\end{axis}
\end{tikzpicture}%
    \end{subfigure}
    \hspace{0.05\textwidth}
    \caption{Loss of bi-orthogonality of the bases generated using the approximate recurrence matrix $\widetilde{H}_{N}$ and the exact recurrence matrix $H_{N}$ for the sequence of Kravchuk (left) and Hahn (right) MOPs.}\label{fig:Hahn_Krav_BO}
    \end{figure}

The results show that even when the exact recurrence matrix $H_{N}$ is used, the bases $V_{N}$ and $W_{N}$ gradually lose biorthogonality as the degree increases. Furthermore, the recurrence coefficients contained in $\widetilde{H}_{N}$ produce $\widetilde{V}_{N}$ and $\widetilde{W}_{N}$ whose biorthogonality is comparable to that obtained using the exact recurrence matrix $H_{N}$. This demonstrates that the computed coefficients in $\widetilde{H}_N$ are sufficiently accurate, and that the observed loss of biorthogonality naturally occurs when generating $V_{N}$ and $W_{N}$ from the recurrence relation.

\begin{example}\label{example3}
We consider Chebyshev nodes and equidistant nodes on the interval $[-1,1]$ associated with the random weights $\bm{\alpha}_j$, $j=1,2$, whose entries are independently and uniformly distributed on $(0,1)$, and a random $P_{0}$.
\end{example}  
For MOP systems with nested weights, it is possible to add nodes one by one and update the corresponding recurrence matrices and bi-orthonormal pairs incrementally. Starting from the solution ${H}_2$, $({V}_2, {W}_2)$ of size $2\times 2$, $\ell$ nodes can be added sequentially, yielding after each step the updated recurrence matrices $\widetilde{H}_3, \widetilde{H}_4, \ldots, \widetilde{H}_{2+\ell}$ and bi-orthonormal matrices $(\widetilde{V}_3, \widetilde{W}_3), (\widetilde{V}_4, \widetilde{W}_4), \ldots, (\widetilde{V}_{2+\ell}, \widetilde{W}_{2+\ell})$. For this example, we run the updating procedure starting from the solution of size \(2\times 2\), and subsequently add \(\ell = 98\) nodes. The error metrics introduced above are tracked and reported in Figure~\ref{fig:Chebyshev}. Since the weights are random, all reported results are averaged over $5$ independent runs.
\begin{figure}[!htbp]
   \begin{subfigure}[b]{0.45\textwidth}
        \centering
		\setlength\figureheight{8cm}
		\setlength\figurewidth{13cm}
        \input{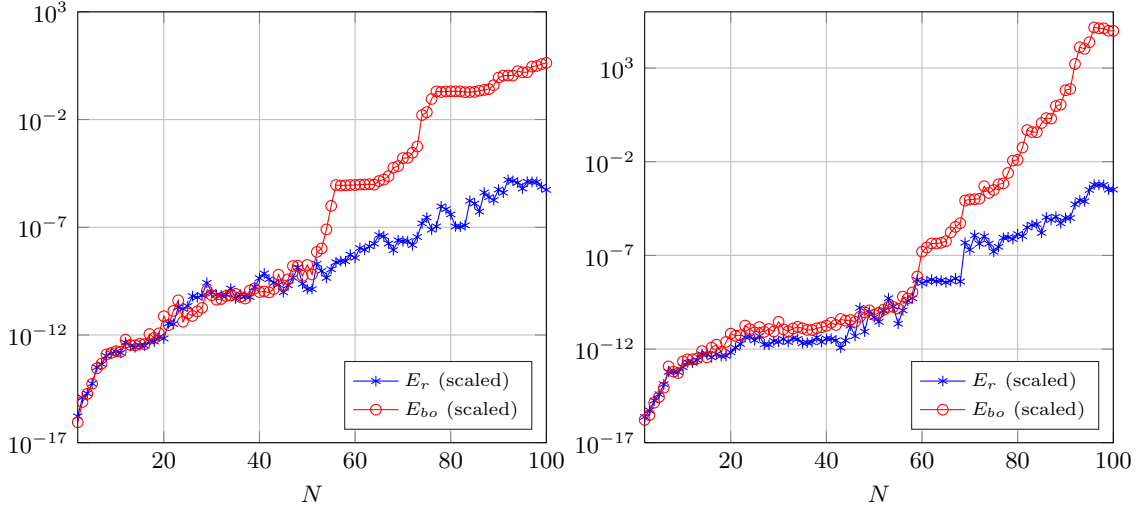}
    \end{subfigure}
    \hspace{0.05\textwidth}
    \caption{Error metrics for increasing the size of the sequence of MOPs, for Chebyshev nodes (left) and equidistant nodes on $[-1,1]$ (right), with random weights, averaged over $5$ runs.}\label{fig:Chebyshev}
    \end{figure}

As shown in Figure~\ref{fig:Chebyshev}, the errors associated with the Chebyshev nodes grow more slowly than those corresponding to the equidistant nodes. In both cases, however, the recurrence relation error and the biorthogonality error remain relatively small, even for large values of $N$. For example, when $N = 100$, the recurrence relation retains approximately five digits of accuracy for the Chebyshev nodes, compared to about three digits for the equidistant nodes.

Although the biorthogonality error increases more rapidly than the recurrence relation error $E_r$, its magnitude remains controlled throughout the experiments.
\subsection{Downdating}
As a first test for the downdating procedure, we consider the setting of Example~\ref{example3}. We first compute the solution ${H}_N$ together with the bi-orthonormal matrices $({V}_N,{W}_N)$ of size $N\times N$ using the updating procedure. Starting from this solution, we then sequentially downdate $\ell$ nodes, obtaining after each step the recurrence matrices $\widetilde{H}_{N-1}, \widetilde{H}_{N-2}, \ldots, \widetilde{H}_{N-\ell}$ and the associated bi-orthonormal matrices $(\widetilde{V}_{N-1}, \widetilde{W}_{N-1}), (\widetilde{V}_{N-2}, \widetilde{W}_{N-2}), \ldots, (\widetilde{V}_{N-\ell}, \widetilde{W}_{N-\ell})$. In this experiment, we take $N=40$ and remove $\ell=20$ nodes successively. The error metrics introduced above are monitored after each downdating step and reported in Figure~\ref{fig:Chebyshev_Downdating}. Since the weights are random, the reported results are averaged over $5$ independent runs.
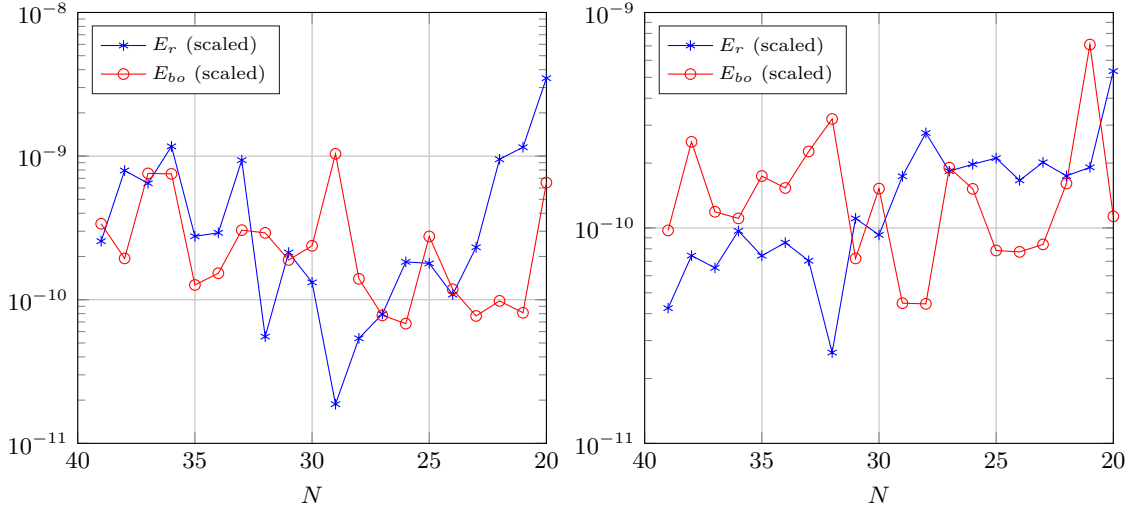
\begin{figure}[H]
   \begin{subfigure}[b]{0.45\textwidth}
        \centering
		\setlength\figureheight{8cm}
		\setlength\figurewidth{13cm}
%
\begin{tikzpicture}

\begin{axis}[%
width=6.2cm,
height=5.7cm,
at={(0cm,0cm)},
scale only axis,
xmin=20,
xmax=40,
x dir=reverse,
xlabel style={font=\color{white!15!black}},
xlabel={$N$},
ymode=log,
ymin=1e-11,
ymax=1e-08,
yminorticks=true,
axis background/.style={fill=white},
xmajorgrids,
ymajorgrids,
yminorgrids,
legend style={
    at={(0.03,0.97)},
    anchor=north west,
    legend cell align=left,
    font=\scriptsize,
    draw=white!15!black
},
grid=major,
ticklabel style={font=\small},
label style={font=\small},
]
\addplot [color=blue, mark=asterisk, mark options={solid, blue}]
  table[row sep=crcr]{%
39	2.55771806081152e-10\\
38	7.93858342127728e-10\\
37	6.46760754472561e-10\\
36	1.169186669839e-09\\
35	2.76582938784314e-10\\
34	2.92304748610525e-10\\
33	9.37166608062324e-10\\
32	5.53742337103853e-11\\
31	2.13061396370267e-10\\
30	1.32107849957588e-10\\
29	1.87425778663784e-11\\
28	5.38833405628622e-11\\
27	7.95228844155525e-11\\
26	1.83477133687243e-10\\
25	1.79103472301396e-10\\
24	1.08630518327019e-10\\
23	2.32057394341035e-10\\
22	9.51059808144307e-10\\
21	1.15275806114458e-09\\
20	3.48667336199498e-09\\
};
\addlegendentry{$ E_r$ (scaled)}

\addplot [color=red, mark=o, mark options={solid, red}]
  table[row sep=crcr]{%
39	3.38141220224825e-10\\
38	1.93944451161938e-10\\
37	7.57976254353771e-10\\
36	7.52359562462699e-10\\
35	1.26723617644773e-10\\
34	1.52914899849558e-10\\
33	3.0507818468088e-10\\
32	2.92057360546751e-10\\
31	1.89218226239802e-10\\
30	2.37452661721391e-10\\
29	1.03717713612835e-09\\
28	1.39842921868112e-10\\
27	7.78056676637895e-11\\
26	6.81820659319445e-11\\
25	2.75776657851224e-10\\
24	1.18299384016058e-10\\
23	7.71093473632276e-11\\
22	9.81977904233448e-11\\
21	8.11297534639916e-11\\
20	6.53977973304016e-10\\
};
\addlegendentry{$E_{bo}$ (scaled)}

\end{axis}

\begin{axis}[%
width=6.2cm,
height=5.7cm,
at={(7.5cm,0cm)},
scale only axis,
xmin=20,
xmax=40,
x dir=reverse,
xlabel style={font=\color{white!15!black}},
xlabel={$N$},
ymode=log,
ymin=1e-11,
ymax=1e-09,
yminorticks=true,
axis background/.style={fill=white},
xmajorgrids,
ymajorgrids,
yminorgrids,
legend style={
    at={(0.03,0.97)},
    anchor=north west,
    legend cell align=left,
    font=\scriptsize,
    draw=white!15!black
},
grid=major,
ticklabel style={font=\small},
label style={font=\small},
]
\addplot [color=blue, mark=asterisk, mark options={solid, blue}]
  table[row sep=crcr]{%
39	4.24360778071901e-11\\
38	7.43839111479957e-11\\
37	6.5282554996889e-11\\
36	9.6831503165662e-11\\
35	7.43105371479676e-11\\
34	8.54797255599395e-11\\
33	7.04212641796335e-11\\
32	2.63885175581788e-11\\
31	1.1072102002937e-10\\
30	9.28073404644155e-11\\
29	1.73771905346664e-10\\
28	2.76033694273216e-10\\
27	1.83555294250856e-10\\
26	1.97524674662492e-10\\
25	2.10901491727733e-10\\
24	1.66337218257324e-10\\
23	2.01623472817451e-10\\
22	1.7439973185031e-10\\
21	1.91195915805258e-10\\
20	5.34649213805311e-10\\
};
\addlegendentry{$ E_r$ (scaled)}

\addplot [color=red, mark=o, mark options={solid, red}]
  table[row sep=crcr]{%
39	9.74052417408492e-11\\
38	2.50934039541573e-10\\
37	1.18853528401463e-10\\
36	1.10724764082226e-10\\
35	1.74106300586794e-10\\
34	1.53273510179762e-10\\
33	2.26532807269464e-10\\
32	3.20593691032706e-10\\
31	7.21827917173701e-11\\
30	1.52257363790841e-10\\
29	4.47390321981702e-11\\
28	4.44227773862947e-11\\
27	1.90437064544949e-10\\
26	1.51816989388769e-10\\
25	7.85933288463753e-11\\
24	7.74258763766593e-11\\
23	8.3859024295132e-11\\
22	1.60868067148335e-10\\
21	7.09806469338302e-10\\
20	1.1317519616047e-10\\
};
\addlegendentry{$E_{bo}$ (scaled)}

\end{axis}

\end{tikzpicture}%
    \end{subfigure}
    \hspace{0.05\textwidth}
    \caption{Error metrics for decreasing the size of the sequence of MOPs, for Chebyshev nodes (left) and equidistant nodes on $[-1,1]$ (right), with random weights, averaged over $5$ runs.}\label{fig:Chebyshev_Downdating}
    \end{figure}
Figure~\ref{fig:Chebyshev_Downdating} shows that the downdating procedure remains numerically stable under successive node removals. The recurrence and biorthogonality errors stay consistently between approximately $10^{-9}$ and $10^{-11}$ for all matrix sizes considered, with no noticeable error growth throughout the downdating process. We emphasize that the downdating process is initialized from the computed $40\times40$ solution obtained by the updating procedure, for which the recurrence and biorthogonality errors remain sufficiently small, ensuring that the starting solution is numerically reliable.

As a second test for the downdating procedure, we again consider the setting of Example~\ref{example3} and begin with the computed solution of size $40 \times 40$. Starting from this solution, we alternately apply one downdating step followed by one updating step. More precisely, at each cycle the last node is first removed, producing a solution of size $39\times39$, and is then added back through the updating procedure to recover a solution of size $40\times40$. This down/update cycle is repeated $20$ times. The recurrence and biorthogonality errors are recorded after every downdating and updating step and are reported in Figure~\ref{fig:Chebyshev_Downdating2}. Since the weights are random, the reported results are averaged over $5$ independent runs.

\begin{figure}[H]
   \begin{subfigure}[b]{0.45\textwidth}
        \centering
		\setlength\figureheight{8cm}
		\setlength\figurewidth{13cm}
        \input{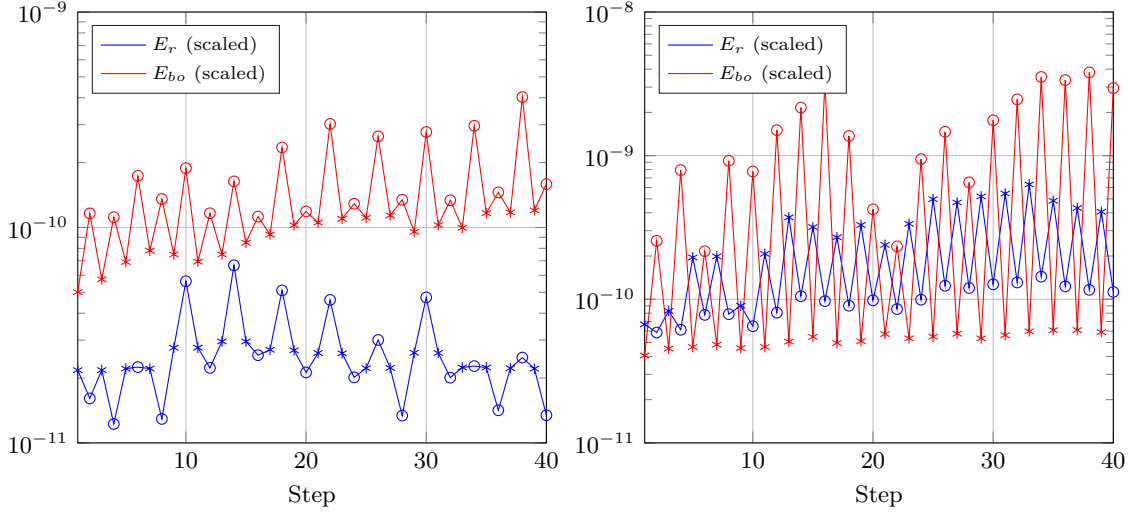}
    \end{subfigure}
    \hspace{0.05\textwidth}
    \caption{Error metrics for the alternating downdating/updating experiment for Chebyshev nodes (left) and equidistant nodes on $[-1,1]$ (right). Marker $*$ corresponds to downdating steps, while $\circ$ corresponds to updating steps.}\label{fig:Chebyshev_Downdating2}
    \end{figure}
    \section{Conclusion and future work}\label{sec:conclusion}
We presented updating and downdating procedures for constructing the step-line recurrence relations of discrete MOPs on the real line by reformulating the problem as a structured Hessenberg inverse eigenvalue problem. The proposed algorithms efficiently modify an existing $(r+2)$-banded upper Hessenberg recurrence matrix when nodes are added to or removed from the underlying discrete measures. Based on structured similarity transformations and eigenvalue deflation, the proposed framework allows existing recurrence relations to be modified without recomputing the entire solution from scratch. Moreover, the introduced scaling strategy improves the conditioning and enables the stable construction of larger recurrence matrices. Numerical experiments confirmed the accuracy, stability, and efficiency of the proposed methods.

Future research will focus on extending the proposed framework to other classes of MOPs, including those associated with recurrence relations arising from different paths of the multi-indices beyond the step-line case, such as nearest-neighbor MOPs \cite{Van-nearestneighbor} or MOPs on the complex plane \cite{MR2358391}. It would also be of interest to develop updating and downdating procedures for more general inner products, for example those involving derivatives of MOPs. Finally, although the proposed algorithms exhibit satisfactory numerical behavior in the experiments presented here, a thorough stability analysis remains an important topic for future investigation.
\section*{Funding}
The research was partially supported by the Research Council KU Leuven (Belgium), project C16/21/002 (Manifactor: Factor Analysis for Maps into Manifolds) and by the Fund for Scientific Research -- Flanders (Belgium), projects G0A9923N (Low rank tensor approximation techniques for up- and downdating of massive online time series clustering) and G0B0123N (Short recurrence relations for rational Krylov and orthogonal rational functions inspired by modified moments).
\section*{Declarations}
\textbf{Conflict of interest:} The authors declare that they have no conflict of interest.
\bibliographystyle{siam}
\bibliography{references}
\end{document}